\documentclass[preprint,11pt]{elsarticle}

\makeatletter
\def\ps@pprintTitle{%
   \let\@oddhead\@empty
   \let\@evenhead\@empty
   \let\@oddfoot\@empty
   \let\@evenfoot\@oddfoot
}
\makeatother

\usepackage{mathrsfs,amssymb,amsmath}
\usepackage{graphicx,amsthm}

\journal{Mathematical Biosciences}

\begin{document}

\newtheorem{theorem}{Theorem}[section]
\newtheorem{proposition}[theorem]{Proposition}
\newtheorem{definition}[theorem]{Definition}
\newtheorem{lemma}[theorem]{Lemma}
\newtheorem{example}[theorem]{Example}
\newtheorem{corollary}[theorem]{Corollary}
\newtheorem{remark}[theorem]{Remark}
\newtheorem{illustration}[theorem]{Illustration}


\begin{frontmatter}


\author{Bryan S. Hernandez\fnref{label2}}

\author{Eduardo R. Mendoza\fnref{label3,label4,label5,label6}}

\author{Aurelio A. de los Reyes V\fnref{label2}}

\title{A Computational Approach to Multistationarity of Power-Law Kinetic Systems}
\address[label2]{Institute of Mathematics, University of the Philippines Diliman, Quezon City 1101, Philippines}
\address[label3]{Institute of Mathematical Sciences and Physics, University of the Philippines Los Ba{\~n}os, Laguna, 4031 Philippines}
\address[label4]{Mathematics and Statistics Department, De La Salle University, Manila, 0922 Philippines}
\address[label5]{Max Planck Institute of Biochemistry, Martinsried, Munich, Germany}
\address[label6]{LMU Faculty of Physics, Geschwister -Scholl- Platz 1, 80539 Munich, Germany}

\author{}

\address{}

\begin{abstract} 

This paper presents a computational solution to determine if a chemical reaction network endowed with power-law kinetics (PLK system) has the capacity for multistationarity, i.e., whether there exist positive rate constants such that the corresponding differential equations admit multiple positive steady states within a stoichiometric class.
The approach, which is called the ``Multistationarity Algorithm for PLK systems'' (MSA), combines (i) the extension of the ``higher deficiency algorithm'' of Ji and Feinberg for mass action to PLK systems with reactant-determined interactions, and (ii) a method that transforms any PLK system to a dynamically equivalent one with reactant-determined interactions. 
Using this algorithm, we obtain two new results: the monostationarity of a popular model of anaerobic yeast fermentation pathway, and the multistationarity of a global carbon cycle model with climate engineering, both in the generalized mass action format of biochemical systems theory. We also provide examples of the broader scope of our approach for deficiency one PLK systems in comparison to the extension of Feinberg's ``deficiency one algorithm'' to such systems.
\begin{keyword}
power-law kinetics \sep higher deficiency algorithm \sep multistationarity \sep anaerobic yeast fermentation pathway \sep global carbon cycle model


\end{keyword}

\end{abstract}

\end{frontmatter}

\section{Introduction}
\label{sec:1}

Biochemical maps are widely used in biochemistry to visualize complex, non-linear interactions between molecules and other chemical entities depicted as nodes. Two types of directed interactions are captured in a biochemical map: transfer of mass (usually represented by solid arrows) from one node to another and regulation (transfer of information, often denoted by dotted line arrows) from a node to a mass transfer arrow. Biochemical maps are akin to the ``box models’’ used in chemical engineering and related disciplines. Figure \ref{yeast} shows the biochemical map for the anaerobic yeast fermentation pathway model of Galazzo and Bailey \cite{galazzo}, and Curto et al. \cite{curto}, which is provided in \cite{voit}.
\begin{figure}
\begin{center}
\label{yeast}
\includegraphics[width=12cm,height=24cm,keepaspectratio]{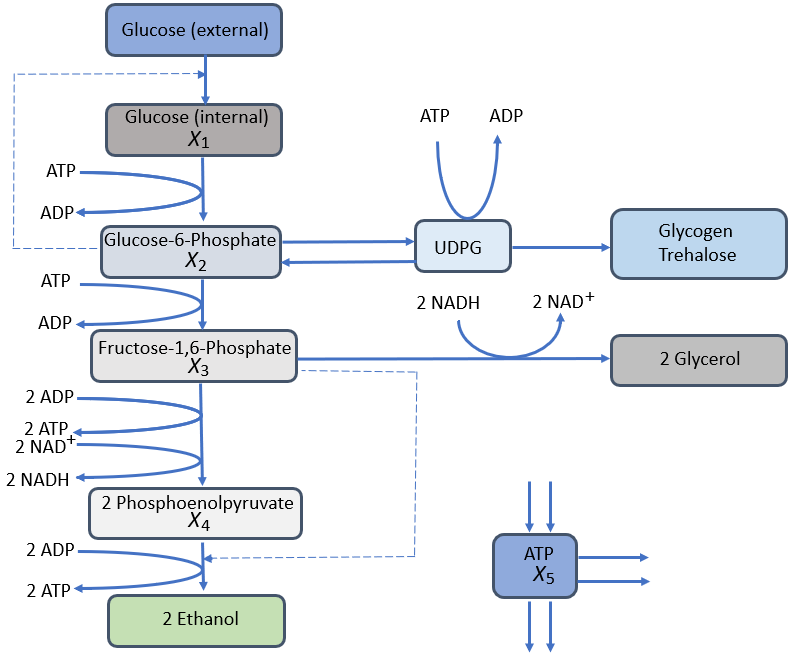}
\caption{Biochemical map of anaerobic fermentation pathway in {\it Saccharomyces cerevisiae} (adapted from \cite{voit}).}
\end{center}
\end{figure}

To describe the dynamics of a biochemical system, an interaction is typically assigned a rate function which captures some of its properties such as speed, type (activating or inhibiting) and strength. In Biochemical Systems Theory (BST), power-law functions, which are products of variables with real exponents, are the rate functions of choice. In the half century, since M. Savageau introduced the BST approach, numerous models in many areas of the life sciences have been developed. There are two main types of BST models: GMA (Generalized Mass Action) systems and S-systems. In a GMA system, a power-law rate function is assigned to each arrow in a biochemical map. In an S-system, which formally is a special GMA system, the incoming (outgoing) arrows of a node are ``lumped’’ together to a single incoming (outgoing) arrow and assigned a power-law function each (often viewed as a “weighted average” of the individual arrows in a corresponding GMA system).

S-systems have many advantages for analysis of system dynamics, including a very simple solution to the question of multistationarity. Determining the positive equilibria is equivalent to setting the two power-law functions equal and taking the logarithm results in a system of linear equations in the logarithms of original variables (i.e., in logarithmic space). The entries of the matrix of this system, which is called the ``A-matrix’’ in BST, are the differences of the corresponding exponents (called ``kinetic orders’’), the nonzero values for which a solution (i.e., an equilibrium) are sought are the ratios of the rate constants.  Linear algebra then tells us when there is a unique, infinitely many or no equilibria. For the structurally more accurate GMA systems however, there are very few general results regarding equilibria addressing the question of multistationarity. One typically could compute results for systems with particular numerical parameters.

The potential of Chemical Reaction Network Theory (CRNT) to provide an approach to the multistationarity problem of GMA systems was a motivating factor for Arceo et al. \cite{arc_jose,arceo} to introduce representations of BST systems as power-law kinetic systems. This involved translating a biochemical map to a full-fledged chemical reaction network and using the concept of ``embedded network’’ to obtain a chemical kinetic system (CKS) dynamically equivalent to the BST system. The analysis of embedded CKS of 15 BST systems (7 of which were GMA systems) revealed that their underlying networks were all non-weakly reversible and have deficiency greater than one (called ``higher deficiency’’ in CRNT). This implied that they did not fall under the coverage of the available results of CRNT on equilibria from Müller-Regensburger \cite{muller} and Talabis et al. \cite{talabis}. In addition, most of the embedded CKS were power-law systems with non-reactant-determined interactions (PL-NDK systems) and the only result for equilibria of PL-NDK systems is a recent Deficiency Zero Theorem from Fortun et al. \cite{fortun3}. The ``Multistationarity Algorithm for PLK systems'' (MSA) presented in this paper provides a computational solution for the multistationarity problem of GMA Systems.

M. Feinberg pioneered the use of computational approaches to multistationarity in chemical reaction networks in 1995 with the novel Deficiency One Algorithm (DOA) for mass action kinetics (MAK) \cite{feinberg_DOA}. It is a significant complement to the Deficiency One Theorem (DOT). In other words, if a reaction network with deficiency one that does not satisfy the conditions of the DOT, one may use the DOA. On the other hand, the higher deficiency algorithm (HDA) for MAK is given in the Ph.D. thesis of H. Ji \cite{ji}. It is a reformulation of the works of Feinberg's DOA \cite{feinberg_DOA,feinberg2} and P. Ellison's Advanced Deficiency Algorithm (ADA) which already handles MAK with higher deficiency \cite{ellison_ADA}. Their results were already implemented in CRNToolbox \cite{software}. When you input a specific reaction network, the software outputs its basic properties. It also reports whether or not the MAK system has the capacity to admit multiple equilibria.

In the MSA, we consider a subset of the reaction set called an {\it orientation}, which induces a partition of the subnetwork into equivalence classes and correspondingly, the whole network is partitioned into fundamental classes. A system of equations and/or inequalities will be obtained and checked whether a solution called {\it signature} exists. If such a signature exists, then the kinetic system has the capacity for multistationarity for particular rate constants. If no signature exists for all possible construction of the system of equations and/or inequalities, then the kinetic system does not have the capacity for multistationarity, no matter what positive values the rate constants assume.

The multistationarity algorithm MSA consists of two new results: the extension of the HDA of Ji and Feinberg for MAK systems to power-law systems with reactant-determined interactions (PL-RDK systems) and its combination with a transformation method in \cite{cfrm} to address the multistationarity of PL-NDK systems. In particular, we present the following new results:
\begin{enumerate}
\item the monostationarity (for any set of rate constants) of the anaerobic yeast fermentation pathway of Curto et al. (Section \ref{shorten:yeast}), and
\item the multistationarity (for certain sets of rate constants) of the global carbon cycle model of Heck et al. (Section \ref{ex:heck}).
\end{enumerate}
In a further new application, we apply the MSA to non-regular deficiency one PL-RDK systems, which are not covered by the extension of Feinberg's DOA to PL-RDK systems by Fortun et al. \cite{fortun}. A deficiency one regular network is required to be $t$-minimal and satisfy a ``cut pair'' condition besides being positive dependent. We also provide an example of a deficiency one PL-NDK system whose transform has deficiency 2, which can still be handled by HDA, but obviously not by DOA.

\begin{remark}
In CRNT, ``higher deficiency'' refers to networks with deficiency greater than one. We desisted from further use of the term ``Higher Deficiency Algorithm'' because the algorithm is valid for any deficiency, and in fact, as remarked above, quite useful for deficiency one networks.
\end{remark}

\section{The Multistationarity Algorithm for Power-Law Systems with Reactant-Determined Interactions}
\label{sect:RDK}

\indent  For power-law kinetic systems with reactant-determined interactions, the multistationarity algorithm MSA is the extension of the HDA of Ji and Feinberg to these systems. Using the discussion of the higher deficiency theory and algorithm for MAK in \cite{ji}, we develop a version for PL-RDK. The algorithm basically follows that of MAK where each reaction $y \to y'$ corresponds to a column in the molecularity matrix $Y$. In PL-RDK, the reaction $y \to y'$ corresponds to a row in the kinetic order matrix and correspondingly, to a column in the $T$-matrix (see Definition \ref{tmatrix} in Appendix \ref{prelim}). With this, we modified the algorithm by considering the $T$-matrix of the kinetic system, for solving system of equations, instead of the molecularity matrix of the reaction network. For a background on CRNT, the reader may refer to Appendix \ref{prelim}.

\subsection{Parameters for Multistationarity}

\indent Let $\left(\mathscr{S},\mathscr{C},\mathscr{R},K\right)$ be a PL-RDK system. If the system admits two positive and distinct equilibria, say $c^{*}$ and $c^{**}$, then $\sum\limits_{y \to y' \in \mathscr{R}} {{k_{y \to y'}}{{\left( {{c^*}} \right)}^{T_{.y}}}\left( {y' - y} \right)} 
= 0$
and
$\sum\limits_{y \to y' \in \mathscr{R}} {{k_{y \to y'}}{{\left( {{c^{**}}} \right)}^{T_{.y}}}\left( {y' - y} \right) = 0}
\label{equi2}$
where $T_{.y}$ is the column of the $T$-matrix associated with the reactant complex $y$.
Lemmas \ref{converse1} and \ref{converse2}, which provide necessary and sufficient conditions for the existence of multiple equilibria, show the importance of the introduction of two new terms $\kappa$ and $\mu$. We modified Lemma 2.2.1 in \cite{ji} to come up with these lemmas.
\begin{lemma} Let $\left(\mathscr{S},\mathscr{C},\mathscr{R},K\right)$ be a PL-RDK system. Suppose there exist a set of positive rate constants $\left\{ {{k_{y \to y'}}|y' \to y \in \mathscr{R}} \right\}$ and two distinct, positive, stoichiometrically compatible ${c^*},{c^{**}}\in \mathbb{R}_{ > 0}^\mathscr{S}$. Then there exist $\kappa  \in \mathbb{R}_{ > 0}^\mathscr{R}$ and $\mu  \in {\mathbb{R}^\mathscr{S}}$ such that ${\mu _s} = \ln \left( {\dfrac{{{c_s}^*}}{{{c_s}^{**}}}} \right) \ \forall s \in \mathscr{S}$ and ${\kappa _{y \to y'}} = {k_{y \to y'}}{c}{^{**T_{.y}}}$ $\forall$ $y \to y' \in \mathscr{R}$
where $\mu$ is both nonzero and stoichiometrically compatible with (the stoichiometric subspace) $S$.
\label{converse1}
\end{lemma}

\begin{lemma} Let $\left(\mathscr{S},\mathscr{C},\mathscr{R},K\right)$ be a PL-RDK system. Suppose  there exist $\kappa  \in \mathbb{R}_{ > 0}^\mathscr{R}$ and $0 \ne \mu  \in {\mathbb{R}^\mathscr{S}}$ which is stoichiometrically compatible with $S$. Then there exist a set of positive rate constants $\left\{ {{k_{y \to y'}}|y' \to y \in \mathscr{R}} \right\}$ and two distinct, positive, stoichiometrically compatible ${c^*},{c^{**}}\in \mathbb{R}_{ > 0}^\mathscr{S}$.
\label{converse2}
\end{lemma}

The first two equations in this section can be expressed as
$$
\sum\limits_{y \to y' \in {\mathscr{R}}} {{\kappa _{y \to y'}}\left( {y' - y} \right) = 0} 
\label{equi3}
{\rm{ \ and \ }}
\sum\limits_{y \to y' \in {\mathscr{R}}} {{\kappa _{y \to y'}}{e^{{T_{.y }\cdot \mu}}}\left( {y' - y} \right) = 0} 
\label{equi4}
$$
where $\mu  = \ln \left( {\dfrac{{{c^*}}}{{{c^{**}}}}} \right)$. The equivalence can be established by letting ${\mu _s} = \ln \left( {\dfrac{{{c_s}^*}}{{{c_s}^{**}}}} \right) \ \forall s \ \in \mathscr{S}$. Hence, ${e^{{\mu _s}}} = \dfrac{{{c_s}^*}}{{{c_s}^{**}}}\;\forall s\; \in \mathscr{S}$.
Then,
$${e^{{T_{.y}} \cdot \mu }} = \prod\limits_{s\; \in {\mathscr{S}}} {{e^{{\mu _s}{{\left( {T{._y}} \right)}_s}}} = } {\prod\limits_{s\; \in {\mathscr{S}}} {\left( {\dfrac{{{c_s}^*}}{{{c_s}^{**}}}} \right)} ^{{{\left( {T{._y}} \right)}_s}}} = \dfrac{{{c_{}}{{^*}^{{T_{.y}}}}}}{{{c_{}}{{^{**}}^{{T_{.y}}}}}}.$$
Thus, ${c_{}}{^{{*{{T_{.y}}}}}} = {e^{{T_{.y}} \cdot \mu }}{c_{}}{^{{{**}{{T_{.y}}}}}}$ and 
 ${k_{y \to y'}}{c_{}}{^{*{{{T_{.y}}}}}} = {k_{y \to y'}}{e^{{T_{.y}} \cdot \mu }}{c_{}}{^{{**}{{T_{.y}}}}} = {\kappa _{y \to y'}}{e^{{T_{.y}} \cdot \mu }}$
since we set ${\kappa _{y \to y'}} = {k_{y \to y'}}{c}{^{**T_{.y}}}$ $\forall$ $y \to y' \in \mathscr{R}$.

\subsection{Fundamental Classes of Reactions}

In this subsection, we introduce a subset of the reaction set $\mathscr{R}$ known as an ``orientation'', denoted by $\mathscr{O}$, which will lead to the ``unique'' partition of $\mathscr{R}$.
\begin{definition} A subset $\mathscr{O}$ of $\mathscr{R}$ is said to be an {\bf orientation} if for every reaction $y \to y' \in \mathscr{R}$, either $y \to y' \in \mathscr{O}$ or $y' \to y \in \mathscr{O}$, but not both.
\end{definition}
For an orientation $\mathscr{O}$, we define a linear map ${L_\mathscr{O}}:{\mathbb {R}^\mathscr{O}} \to S$ such that
\[{L_\mathscr{O}}(\alpha)= \sum\limits_{y \to y' \in \mathscr{O}} {{\alpha _{y \to y'}}\left( {y' - y} \right)}.\]

\begin{remark}
If the network $\mathscr{R}$ has no reversible reactions, then there is only one orientation $\mathscr{O}=\mathscr{R}$. In this case, $L_\mathscr{R} = YI_a = N$, the stoichiometric matrix.
\end{remark}

\begin{illustration}
\label{sacc:cer}
Consider the GMA model of anaerobic fermentation pathway of {\it Saccharomyces cerevisiae} in Chapter 8 of \cite{voit} depicted in Figure \ref{yeast}. The model has the following reactions.
\[ \begin{array}{lll}
R_1: X_2 \to X_1 + X_2 & \ \ \ &  R_8: X_3 + X_5  \to X_4 + X_5\\
R_2: X_1 + X_5  \to X_2 + X_5  &  \ \ \ & R_9: X_3 + X_5 \to X_3 + 2X_5 \\
R_3: 2X_5 + X_1  \to X_5 + X_1 &  \ \ \ & R_{10}: X_3 + X_4 + X_5 \to X_4 + X_5\\
R_4: X_2 + X_5  \to X_3 + X_5  &  \ \ \ & R_{11}: X_3 + X_4 + X_5 \to X_3 + X_5  \\
R_5: 2X_5 + X_2  \to X_5 + X_2 & \ \ \  & R_{12}: X_3 + X_4 + X_5  \to X_3 + X_4 + 2X_5\\
R_6: X_2 + X_5  \to X_5  &  \ \ \ & R_{13}: 2X_5 \to X_5 \\
R_7: X_2 + X_5  \to X_2  &  \\
\end{array}\]
In \cite{arceo}, the system was referred to as ERM0-G and it was shown that the system is PL-RDK. Each reaction is irreversible. Thus, ${\mathscr{O}}=\mathscr{R}.$ 
\end{illustration}

Let $\left\{ {{\omega _{y \to y'}}|y \to y' \in {\mathbb{R}^\mathscr{O}}} \right\}$ be the standard basis for ${{\mathbb{R}^\mathscr{O}}}$. Reactions ${y \to y'}$, $\overline y  \to \overline y '\in \mathscr{O}$ are related by $\sim$ if there exists a nonzero $\alpha$ such that ${\omega _{y \to y'}} - \alpha {\omega _{\overline y  \to \overline y '}} \in Ker ^{\perp}L_\mathscr{O}$. It is easy to see that the relation $\sim$ is indeed an equivalence relation. We define the {\bf equivalence class} of a reaction $y \to y'$ under the relation $\sim$ by $\left[y \to y' \right]=\{{\widetilde y  \to \widetilde y '}\in \mathscr{O}|{\widetilde y  \to \widetilde y '} \sim {y \to y' }\}$. We call the equivalence classes as $P_i$ where $i$ runs through the number of classes induced by the relation. In particular, if there is a reaction such that ${\omega _{y \to y'}} \in Ker ^{\perp}L_\mathscr{O}$, we label the equivalence class containing $y \to y'$ as $P_0$. We set $P_0=\emptyset$ if there is no such equivalence class exists. 

\begin{lemma} {\cite{ji}} Let $\left(\mathscr{S},\mathscr{C},\mathscr{R}\right)$ be reaction network and $\mathscr{O}$ be an orientation. Let $\alpha \ne 0$. Then the following holds:
\begin{itemize}
\item[i.] $y \to y' \in {P_0}$
if and only if ${x_{y \to y'}} = 0$ for each $x\in KerL_\mathscr{O}$, and
\item[ii.] ${\omega _{y \to y'}} - \alpha {\omega _{\overline y  \to \overline y '}} \in Ker ^{\perp}L_\mathscr{O}$ if and only if ${x_{y \to y'}} = \alpha {x_{\overline y  \to \overline y '}} \ \forall x \in Ker L_\mathscr{O}$.
\end{itemize}
\label{equiv_classes1}
\end{lemma}
Lemma \ref{equiv_classes1} leads to partitioning the orientation $\mathscr{O}$ into the zeroth equivalence class $P_0$, if it exists, and the nonzeroth equivalence classes $P_i$ where $i \ge 1$. 
\begin{remark} \cite{ji}
\label{rem:equivclass}
Let $\left\{ {{v^l}} \right\}_{l = 1}^d$ be a basis for $Ker{L_\mathscr{O}}$.
If for $y \to y' \in \mathscr{O}$, $v_{y \to y'}^l =0$ for all $1 \le l \le d$ then the reaction $y \to y' $ belongs to the zeroth equivalence class $P_0$.
For $y \to y', {\overline y  \to \overline y '} \in \mathscr{O} \backslash P_0$, if there exists $\alpha \ne 0$ such that $v_{y \to y'}^l = \alpha v_{\overline y  \to \overline y '}^l$ for all $1 \le l \le d$, then the two reactions belong to the same equivalence class.
\end{remark}
The reactions $y \to y'$ and $\overline y  \to \overline y '$ in $\mathscr{R}$ belong to the same {\bf fundamental class} if at least one of the following statements holds \cite{ji}.
\begin{itemize}
\item[i.] $y \to y'$ and $\overline y  \to \overline y '$ are the same reaction.
\item[ii.] $y \to y'$ and $\overline y  \to \overline y '$ are reversible pair.
\item[iii.] Either $y \to y'$ or $y' \to y$, and either $\overline y  \to \overline y '$ or $\overline y'  \to \overline y $ are in the same equivalence class on $\mathscr{O}$.
\end{itemize}

We label the fundamental class containing $P_i$ as $C_i$. A fundamental (equivalence) class is said to be {\bf reversible} if each reaction in the fundamental (equivalence) class is reversible (with respect to $\mathscr{R}$). If at least one of the reactions is irreversible, then the class is said to be {\bf nonreversible}. If $P_i$ is nonreversible, we pick any irreversible reaction to be the representative of $P_i$. Otherwise, we pick any irreversible reaction as the representative. We identify the $i$th reaction in $W$ with ${y_i} \to {y_i}'$. We let $W$ be the collection of all such representatives from $P_i$ where $i=1,2,...,w$.

\begin{illustration}
\label{sacc:cer2}
We consider ERM0-G  in Illustration \ref{sacc:cer}.
Below is basis for $Ker{L_\mathscr{O}}$ obtained by solving $\sum\limits_{y \to y' \in \mathscr{O}} {{\alpha _{y \to y'}}\left( {y' - y} \right)}=0$.
\[\bordermatrix{%
  & v_1 & v_2 & v_3 & v_4 & v_5 &  v_6 & v_7 & v_8\cr
R_1 & 0 & 1 & 0  & 0 & 1 & 1& 0 & 0\cr
R_2 & 0 & 1 & 0  & 0 & 1 & 1& 0 & 0\cr
R_3 & -1 & 0 & -1  & 1 & 0 & 0& 1 & -1\cr
R_4 & 0 & 0 & 0  & 0 & 1 & 1& 0 & 0\cr
R_5 & 1 & 0 & 0  & 0 & 0 & 0& 0 & 0\cr
R_6 & 0 & 1 & 0  & 0 & 0 & 0& 0 & 0\cr
R_7 & 0 & 0 & 1  & 0 & 0 & 0& 0 & 0\cr
R_8 & 0 & 0 & 0  & 0 & 0 & 1& 0 & 0\cr
R_9 & 0 & 0 & 0  & 1 & 0 & 0& 0 & 0\cr
R_{10} & 0 & 0 & 0  & 0 & 1 & 0& 0 & 0\cr
R_{11} & 0 & 0 & 0  & 0 & 0 & 1& 0 & 0\cr
R_{12} & 0 & 0 & 0  & 0 & 0 & 0& 1 & 0\cr
R_{13} & 0 & 0 & 0  & 0 & 0 & 0& 0 & 1\cr
}\]
Notice that there is no row with zero entry so $P_0$ is empty. We partition $\mathscr{O}$ into its equivalence classes.
The equivalence classes and the fundamental classes based on the given basis for $Ker{L_\mathscr{O}}$ are provided in the first and second columns of Table \ref{tab:PiCiEMR0G}, respectively.
\begin{table}
\caption{Nonterminal and terminal strong linkage classes of ERM0-G}
\label{tab:PiCiEMR0G}       
\begin{tabular}{lllll}
\hline\noalign{\smallskip}
Equivalence & Fundamental & Nonterminal Strong & Terminal Strong\\
Class & Class & Linkage Class & Linkage Class\\
\noalign{\smallskip}\hline\noalign{\smallskip}
$P_1 =\{ R_1,R_2\}$ & $C_{1}=\{ R_1,R_2\}$ & $\{X_{2}\}$ & $\{X_1 + X_{2} \}$\\
 &  & $\{X_{1} + X_{5}\}$ & $\{X_2 + X_{5} \}$\\
$P_2 =\{ R_3\}$ & $C_{2}=\{ R_3\}$ & $\{2X_{5} + X_{1}\}$ & $\{X_{5} + X_{1}\}$\\
$P_3 =\{ R_4\}$ & $C_{3}=\{ R_4\}$ & $\{X_2+ X_5\}$ & $\{X_3+ X_5 \}$\\
$P_4 =\{ R_5\}$ & $C_{4}=\{ R_5\}$ & $\{2X_5+X_2 \}$ & $\{ X_5+X_2  \}$\\
$P_5 =\{ R_6\}$ & $C_{5}=\{ R_6\}$ & $\{X_2+X_5 \}$ & $\{X_5 \}$\\
$P_6 =\{ R_7\}$ & $C_{6}=\{ R_7\}$ & $\{ X_2 +X_5\}$ & $\{ X_2  \}$\\
$P_7 =\{ R_8,R_{11}\}$ & $C_{7}=\{ R_8,R_{11}\}$ & $\{ X_3+X_5  \}$ & $\{ X_4 + X_5  \}$\\
 &  & $\{ X_3+X_4+X_5  \}$ & $\{ X_3 + X_5  \}$\\
$P_8 =\{ R_9\}$ & $C_{8}=\{ R_9\}$ & $\{ X_3+X_5 \}$ & $\{ X_3 + 2X_5  \}$\\
$P_9 =\{ R_{10}\}$ & $C_{9}=\{ R_{10}\}$ & $\{ X_3+X_4+X_5 \}$ & $\{ X_4+X_5  \}$\\
$P_{10} =\{ R_{12}\}$ & $C_{10}=\{ R_{12}\}$ & $ \{ X_3+X_4+X_5 \}$ & $\{ X_3+X_4+2X_5  \}$\\
$P_{11} =\{ R_{13}\}$ & $C_{11}=\{ R_{13}\}$ & $ \{2X_5 \}$ & $\{ X_5 \}$\\
\noalign{\smallskip}\hline\noalign{\smallskip}
\end{tabular}
\end{table}
\end{illustration}

A fundamental class $C_i$ with $0 \le i \le w$ is said to be degenerate if $g_{y_i\to y_{i}'}=0$ while a fundamental class $C_i$ with $1 \le i \le w$ is said to be nondegenerate if $g_{y_i\to y_{i}'}\ne 0$. Let $i \ge 1$. For each nondegenerate fundamental class $C_i$, we assume a 3-shelf bookcase to store all reactions in $C_i$. Let $y \to y'$ be a reaction in a nondegenerate fundamental class $C_i$, and 
${\rho _{{y_i} \to {y_i}'}} = \dfrac{{{h_{{y_i} \to {y_i}'}}}}{{{g_{{y_i} \to {y_i}'}}}}$ where ${g_{{y_i} \to {y_i}'}} \ne 0$ and $i = 1,...,w$.
 Then we define the shelving of $y \to y'$ in the following manner.
\begin{itemize}
\item[i.] $y \to y'$ is on the upper shelf if ${e^{{T_{.y}} \cdot \mu }} > \rho_{y_i\to y_{i}'}$.
\item[ii.] $y \to y'$ is on the lower shelf if ${e^{{T_{.y}} \cdot \mu }} < \rho_{y_i\to y_{i}'}$.
\item[iii.] $y \to y'$ is on the middle shelf if ${e^{{T_{.y}} \cdot \mu }} = \rho_{y_i\to y_{i}'}$.
\end{itemize}

Let ${M_{{y_i} \to {y_i}'}} = \ln {\rho _{{y_i} \to {y_i}'}}$ if ${\rho _{{y_i} \to {y_i}'}}>0$. Otherwise, we take ${M_{{y_i} \to {y_i}'}}$ to be an arbitrarily large and negative number.
We restrict the problem by considering a representative of each of the classes instead of all the elements of the whole class. Lemmas \ref{rev_irrev4} and \ref{rev_irrev2s} are extensions of Lemma 2.8.7 and Proposition 2.8.1 in {\cite{ji}} from MAK to PL-RDK.

\begin{lemma} Let $\left(\mathscr{S},\mathscr{C},\mathscr{R},K\right)$ be a PL-RDK system and $\mathscr{O}$ be an orientation. Let $\kappa  \in \mathbb{R}_{ > 0}^\mathscr{R}$, and $\mu  \in {\mathbb{R}^\mathscr{S}}$. Let $g,h \in \mathbb {R}^\mathscr{O}$ such that\\
${g_{y \to y'}} = \left\{ \begin{array}{ll}
{\kappa _{y \to y'}} - {\kappa _{y' \to y}}&{\rm{   if }} \ y \to y' \in \mathscr{O}{\rm{\  is \ reversible}}\\
{\kappa _{y \to y'}}&{\rm{          if }} \ y \to y' \in \mathscr{O}{\rm{ \  is \  irreversible}}
\end{array} \right.$
and\\
${h_{y \to y'}} = \left\{ \begin{array}{ll}
{\kappa _{y \to y'}}{e^{{T_{.y}} \cdot \mu }} - {\kappa _{y' \to y}}{e^{{T_{.y'}} \cdot \mu }}&{\rm{  if }} \ y \to y' \in \mathscr{O}{\rm{\ is \ reversible}}\\
{\kappa _{y \to y'}}{e^{{T_{.y}} \cdot \mu }}&{\rm{           if }} \ y \to y' \in \mathscr{O}{\rm{\ is \ irreversible}}
\end{array} \right..$\\
For $i=1,2,...,w$, let ${P_i}$ be the equivalence class with ${y_i} \to {y_i}'$ as representative. Moreover, let ${{\rho _{{y_i} \to {y_i}'}} = \dfrac{{{h_{{y_i} \to {y_i}'}}}}{{{g_{{y_i} \to {y_i}'}}}}}$ for nondegenerate fundamental classes ${C_i}$.
\begin{itemize}
\item[i.] If $y \to y' \in {P_i}\ (i=1,2,...,w)$ is irreversible then ${g_{y \to y'}} > 0$, ${h_{y \to y'}} > 0$, and ${M_{{y_i} \to {y_i}'}} = {{{T_{.y}} \cdot \mu }}$.
\item[ii.] Suppose $y \to y' \in {P_i} \ (i=1,2,...,w)$ is reversible and ${C_i}$ is nondegenerate.
\begin{itemize}
\item[a.] If ${g_{y_i \to y_i'}} > 0$ and $y \to y'$ is on the upper shelf, then ${M_{{y_i} \to {y_i}'}} < {{{T_{.y}} \cdot \mu }} < {{{T_{.y'}} \cdot \mu }}$.
\item[b.] If ${g_{y_i \to y_i'}} > 0$ and $y \to y'$ is on the middle shelf, then ${M_{{y_i} \to {y_i}'}} = {{{T_{.y}} \cdot \mu }} = {{{T_{.y'}} \cdot \mu }}$.
\item[c.] If ${g_{y_i \to y_i'}} > 0$ and $y \to y'$ is on the lower shelf, then ${M_{{y_i} \to {y_i}'}} > {{{T_{.y}} \cdot \mu }} > {{{T_{.y'}} \cdot \mu }}$.
\item[d.] If ${g_{y_i \to y_i'}} < 0$ and $y \to y'$ is on the upper shelf, then ${M_{{y_i} \to {y_i}'}} < {{{T_{.y'}} \cdot \mu }} < {{{T_{.y}} \cdot \mu }}$.
\item[e.] If ${g_{y_i \to y_i'}} < 0$ and $y \to y'$ is on the middle shelf, then ${M_{{y_i} \to {y_i}'}} = {{{T_{.y'}} \cdot \mu }} = {{{T_{.y}} \cdot \mu }}$.
\item[f.] If ${g_{y_i \to y_i'}} < 0$, and $y \to y'$ is on the lower shelf, then ${M_{{y_i} \to {y_i}'}} > {{{T_{.y'}} \cdot \mu }} > {{{T_{.y}} \cdot \mu }}$.
\end{itemize}
\item[iii.] Suppose $y \to y' \in {P_i}\ (i=1,2,...,w)$ is reversible and ${C_i}$ is degenerate.
\begin{itemize}
\item[a.] If ${h_{y_i \to y_i'}} > 0$ then ${{{T_{.y}} \cdot \mu }} > {{{T_{.y'}} \cdot \mu }}$.
\item[b.] If ${h_{y_i \to y_i'}} = 0$ then ${{{T_{.y}} \cdot \mu }} = {{{T_{.y'}} \cdot \mu }}$.
\item[c.] If ${h_{y_i \to y_i'}} < 0$ then ${{{T_{.y}} \cdot \mu }} < {{{T_{.y'}} \cdot \mu }}$.
\end{itemize}
\item[iv.] If $y \to y' \in {P_0}$ is reversible then ${{{T_{.y}} \cdot \mu }} = {{{T_{.y'}} \cdot \mu }}$.
\end{itemize}
\label{rev_irrev4}
\end{lemma}

\begin{lemma} Let $\left(\mathscr{S},\mathscr{C},\mathscr{R},K\right)$ be a PL-RDK system and $\mathscr{O}$ be an orientation. Suppose there exist $\mu  \in {\mathbb{R}^\mathscr{S}}$, $g,h \in KerL_\mathscr{O}$, ${P_i} \ (i=1,2,...,w)$ with representative ${{y_i} \to {y_i}'}$
and $\left\{ {{\rho _{{y_i} \to {y_i}'}} = \dfrac{{{h_{{y_i} \to {y_i}'}}}}{{{g_{{y_i} \to {y_i}'}}}}|{g_{{y_i} \to {y_i}'}} \ne 0,i = 1,...,w} \right\}$ satisfy the conditions given in Lemma \ref{rev_irrev4}. Then the following holds for a nondegenerate $C_i$.
  \begin{itemize}
\item[i.] All irreversible reactions in $C_i \ (i\ge 1)$ must belong to the middle shelf.
\item[ii.] $y \to y' \in {C_i}$ must belong to the upper shelf if ${{\rho _{{y_i} \to {y_i}'}}}\le 0$.
\item[iii.] If a reaction is reversible, then the reaction and its reversible pair must belong to the same shelf.
\item[iv.] Any two reactions in $C_i$ with the same reactant complex must belong to the same shelf.
\item[v.] Each reaction whose reactant complex lies in a nonterminal strong linkage class must belong to the middle shelf.
\item[vi.] Each reaction whose reactant complex lies in a terminal strong linkage class of the fundamental subnetwork must belong to the same shelf.
\item[vii.] If for a nondegenerate $C_i \ (i\ge 1)$, $\mathscr{N}_i$ forms a big (undirected) cycle (with at least three vertices), then its reactions are all in a terminal strong linkage class and belong to the middle shelf, where $\mathscr{N}_i$ is the subnetwork generated by reactions in $P_i$.
  \end{itemize}
\label{rev_irrev2s}
\end{lemma}

\begin{illustration}
\label{sacc:cer3}
Again, consider ERM0-G  in Illustration \ref{sacc:cer}. Since irreversible reactions belong to the middle shelf, we obtain the following shelving assignment.\\
${{\cal U}_1} = \left\{ {} \right\},{{\cal M}_1} = \left\{ { R_1 : X_{2} \to X_1 + X_{2},R_2 : X_{1}+X_5 \to X_2 + X_{5}    } \right\},{{\cal L}_1} = \left\{ {} \right\}$\\
${{\cal U}_2} = \left\{ {} \right\},{{\cal M}_2} = \left\{ { R_3: 2X_{5} + X_{1}  \to X_5 + X_{1}   } \right\},{{\cal L}_2} = \left\{ {} \right\}$\\
${{\cal U}_3} = \left\{ {} \right\},{{\cal M}_3} = \left\{ {  R_4 : X_2+X_5  \to X_3+X_5  } \right\},{{\cal L}_3} = \left\{ {} \right\}$\\
${{\cal U}_4} = \left\{ {} \right\},{{\cal M}_4} = \left\{ {  R_5: 2X_{5} + X_{2}  \to X_5 + X_{2} } \right\},{{\cal L}_4} = \left\{ {} \right\}$\\
${{\cal U}_5} = \left\{ {} \right\},{{\cal M}_5} = \left\{ {   R_6 : X_2+X_5  \to X_5 } \right\},{{\cal L}_5} = \left\{ {} \right\}$\\
${{\cal U}_6} = \left\{ {} \right\},{{\cal M}_6} = \left\{ {  R_7 : X_2+X_5  \to X_2 } \right\},{{\cal L}_6} = \left\{ {} \right\}$\\
${{\cal U}_7} = \left\{ {} \right\},{{\cal M}_7} = \left\{ {  R_8: X_3+X_5  \to X_4 + X_5,  R_{11}: X_3+X_4+X_5  \to X_3 + X_5} \right\},\\{{\cal L}_7} = \left\{ {} \right\}$\\
${{\cal U}_8} = \left\{ {} \right\},{{\cal M}_8} = \left\{ { R_9: X_3+X_5 \to X_3 + 2X_5   } \right\},{{\cal L}_8} = \left\{ {} \right\}$\\
${{\cal U}_9} = \left\{ {} \right\},{{\cal M}_9} = \left\{ { R_{10}: X_3+X_4+X_5  \to X_4 + X_5 } \right\},{{\cal L}_9} = \left\{ {} \right\}$\\
${{\cal U}_{10}} = \left\{ {} \right\},{{\cal M}_{10}} = \left\{ {  R_{12} : X_3+X_4+X_5 \to X_3+X_4+2X_5  } \right\},{{\cal L}_{10}} = \left\{ {} \right\}$\\
${{\cal U}_{11}} = \left\{ {} \right\},{{\cal M}_{11}} = \left\{ { R_{13} : 2X_5  \to  X_5   } \right\},{{\cal L}_{11}} = \left\{ {} \right\}$
\end{illustration}

\subsection{Sign Patterns and the Fundamental Theorem of Multistationarity in PL-RDK systems}

Suppose $g,h \in Ker{L_\mathscr{O}}$ such that $g \ne 0$ but $h = 0$. For any nonzero sign pattern which is stoichiometrically compatible with $Ker{L_\mathscr{O}}$, we have a solution of nonzero $g \in Ker{L_\mathscr{O}}$ with such sign patterns. Define $${\Gamma _W} = \left\{ {x \in {\mathbb{R}^\mathscr{O}}|x{\rm{ \ has \ support \ in \ }}W} \right\}$$ with $W = \left\{ {{y_i} \to {y_i}'|i = 1,...,w} \right\} \in \mathscr{O}$.

We now focus our attention in depicting what is meant by a ``valid'' pair of sign patterns for ${g_W}=g|_W,{h_W}=h|_W \in  {\mathbb{R}^\mathscr{O}} \cap \Gamma _W$. A pair of sign patterns for ${g_W},{h_W} \in  {\mathbb{R}^\mathscr{O}} \cap \Gamma _W$ is said to be compatible with $Ker{L_\mathscr{O}}{|_W}$ if both of these statements hold.
\begin{itemize}
\item[i.] If ${C_i}\left( {i = 1,2,...,w} \right)$ is nonreversible then the signs of ${g_W}\left( {y \to y'} \right)$ and ${h_W}\left( {y \to y'} \right)$ are both positive.
\item[ii.] For every reaction ${y \to y'}$ in $W$, $v_{{y_i} \to {y_i}'}^1$ and ${g_W}\left( {y \to y'} \right)$ have the same sign, and $v_{{y_i} \to {y_i}'}^2$ and ${h_W}\left( {y \to y'} \right)$ have the same sign, for some ${v^1},{v^2} \in Ker{L_\mathscr{O}}$.
\end{itemize}
In other words, a pair of sign patterns for ${g_W}$ and ${h_W}$ are said to be valid, if it is nonzero and stoichiometrically compatible with $Ker{L_\mathscr{O}}{|_W}$.

At this point, we assume that there is a sign pattern for ${g_W} \in {\mathbb{R}^\mathscr{O}} \cap {\Gamma _W}$. Define the sets
\[D = \left\{ {{y_i} \to {y_i}' \in W|{g_W}\left( {{y_i} \to {y_i}'} \right) = 0} \right\}\]
and
\[ND = \left\{ {{y_i} \to {y_i}' \in W|{g_W}\left( {{y_i} \to {y_i}'} \right) \ne 0} \right\}.\]

Let $\{b^j\}_{j=1}^q$ be a basis for $Ker^\perp {L_\mathscr{O}}\cap {\Gamma _W}$, if it exists, where $q$ is the dimension of $Ker^\perp {L_\mathscr{O}}\cap {\Gamma _W}$. For $j = 1,2,...,q$, we also define the following sets and consider the following equations:
$$
\begin{aligned}
R_ + ^j =& \left\{ {{y_i} \to {y_i}' \in ND|b_{{y_i} \to {y_i}'}^j{g_W}\left( {{y_i} \to {y_i}'} \right) > 0} \right\}\\
R_ - ^j =& \left\{ {{y_i} \to {y_i}' \in ND|b_{{y_i} \to {y_i}'}^j{g_W}\left( {{y_i} \to {y_i}'} \right) < 0} \right\}\\
Q_ + ^j =& \left\{ {{y_i} \to {y_i}' \in D|b_{{y_i} \to {y_i}'}^j{g_W}\left( {{y_i} \to {y_i}'} \right) > 0} \right\}\\
Q_ - ^j =& \left\{ {{y_i} \to {y_i}' \in D|b_{{y_i} \to {y_i}'}^j{g_W}\left( {{y_i} \to {y_i}'} \right) < 0} \right\}\\
Q_1^j =& \left\{ {{\rho _W}\left( {{y_i} \to {y_i}'} \right)|{y_i} \to {y_i}' \in R_ + ^j} \right\}\\
Q_2^j =& \left\{ {{\rho _W}\left( {{y_i} \to {y_i}'} \right)|{y_i} \to {y_i}' \in R_ - ^j} \right\}
\end{aligned}
$$
\begin{equation}
\label{eq:Ker3}
\sum\limits_{{y_i} \to {y_i}' \in ND} {b_{{y_i} \to {y_i}'}^j{g_W}\left( {{y_i} \to {y_i}'} \right) = 0},\qquad j=1,2,...,q
\end{equation}
\begin{equation}
\label{eq:Ker4}
\begin{aligned}
\sum\limits_{{y_i} \to {y_i}' \in ND} {{\rho _W}\left( {{y_i} \to {y_i}'} \right)b_{{y_i} \to {y_i}'}^j{g_W}\left( {{y_i} \to {y_i}'} \right)} \\
+\sum\limits_{{y_i} \to {y_i}' \in D} b_{{y_i} \to {y_i}'}^j{h_W}\left( {{y_i} \to {y_i}'} \right) = 0, \qquad j=1,2,...,q
\end{aligned}
\end{equation}
\begin{equation}
\label{eq:Ker5}
{h_W}\left( {{y_i} \to {y_i}'} \right) = {\rho _W}\left( {{y_i} \to {y_i}'} \right){g_W}\left( {{y_i} \to {y_i}'} \right),\qquad {y_i} \to {y_i}' \in ND.
\end{equation}

\begin{lemma} {\cite{ji}} 
Suppose a reaction network satisfies the following properties for an orientation $\mathscr{O}$: $P_i$ ($i=0,1,2,...,w$) is defined by a representative ${{y_i} \to {y_i}'}$, $W=\{{{y_i} \to {y_i}'}|i=1,2,...,w\}\subseteq \mathscr{O}$, a given basis $\left\{ {{b^j}} \right\}_{j = 1}^q$ for $Ker^\perp {L_\mathscr{O}}\cap {\Gamma _W}$, and a set of parameters\\$\left\{ {{\rho _W}\left( {{y_i} \to {y_i}'} \right)|{g_W}\left( {{y_i} \to {y_i}'} \right) \ne 0,i = 1,2,...,w} \right\}$ where the sign of ${{\rho _W}\left( {{y_i} \to {y_i}'} \right)}$ is the same as the ratio of the signs of ${{h_W}\left( {{y_i} \to {y_i}'} \right)}$ and ${{g_W}\left( {{y_i} \to {y_i}'} \right)}$.
Further, suppose there exist ${h_W},{g_W} \in {\mathbb{R}^\mathscr{O}} \cap {\Gamma _W}$ with a valid pair of sign patterns such that Equations (\ref{eq:Ker3}), (\ref{eq:Ker4}), and (\ref{eq:Ker5}) are satisfied. Then the following holds for $j=1,2,...,q$:
\begin{itemize}
\item[i.] If $\sum\limits_{{y_i} \to {y_i}' \in D} {b_{{y_i} \to {y_i}'}^j{h_W}\left( {{y_i} \to {y_i}'} \right) > 0} $, then one element in $Q_2^j$ is strictly greater than one element in $Q_1^j$.
\item[ii.] If $\sum\limits_{{y_i} \to {y_i}' \in D} {b_{{y_i} \to {y_i}'}^j{h_W}\left( {{y_i} \to {y_i}'} \right) < 0} $, then one element in $Q_1^j$ is strictly greater than one element in $Q_2^j$.
\item[iii.] If $\sum\limits_{{y_i} \to {y_i}' \in D} {b_{{y_i} \to {y_i}'}^j{h_W}\left( {{y_i} \to {y_i}'} \right) = 0} $, then $Q_1^j$ and $Q_2^j$ are nonsegregated. That is, at least one of these holds:
\begin{itemize}
\item[a.] There exist an element $a$ from one of multisets $Q_1^j$ and $Q_2^j$, and $b<c$ from the other such that a is between $b$ and $c$, i.e., $b<a<c$.
\item[b.] All the elements in the multisets $Q_1^j$ and $Q_2^j$ are equal or there are $a,b \in Q_1^j$ and $c,d \in Q_2^j$ where $c = a < b = d$.
\end{itemize}
\end{itemize}
\label{thm_multisets}
\end{lemma}

\begin{definition}
Let $\left(\mathscr{S},\mathscr{C},\mathscr{R}\right)$ be a chemical reaction network with an orientation $\mathscr{O}$. The chemical reaction network is said to have a forestal property if for given $W \subseteq \mathscr{O}$, $Ker^{\perp}{L_\mathscr{O}} \cap {\Gamma _W}$ has a forest basis. $Ker^{\perp}{L_\mathscr{O}} \cap {\Gamma _W}$ has a forest basis if it has a basis such that the graph based on the basis vectors is a forest graph. (The reader may refer to \cite{ji} for further details on forest graphs.)
\end{definition}

The following theorem is an extension of Theorem 2.11.9 in {\cite{ji}} to PL-RDK system which we shall call the ``Fundamental Theorem of Multistationarity in PL-RDK systems''.

\begin{theorem} 
Suppose the reaction network $\left(\mathscr{S},\mathscr{C},\mathscr{R}\right)$ has a forestal property and $\left(\mathscr{S},\mathscr{C},\mathscr{R},K\right)$ is a PL-RDK system. It has the capacity to admit multiple steady states if and only if the following statements hold:
\begin{itemize}
\item[i.] $0 \ne \mu  \in {\mathbb{R}^\mathscr{S}}$ exists which is stoichiometrically compatible with $S$,
\item[ii.] a valid sign pattern for ${g_W},{h_W} \in \mathbb{R}^\mathscr{O} \cap{\Gamma _W}$ exists,
\item[iii.] a set of parameters $\left\{ {{\rho _W}\left( {{y_i} \to {y_i}'} \right)|{g_W}\left( {{y_i} \to {y_i}'} \right) \ne 0} \right\}$ where the sign of ${\rho _W}\left( {{y_i} \to {y_i}'} \right)$ is precisely the ratio of the signs of ${{h_W}\left( {{y_i} \to {y_i}'} \right)}$ and\\${{g_W}\left( {{y_i} \to {y_i}'} \right)}$ that satisfies the conditions in Lemma \ref{thm_multisets}, and
\item[iv.] a shelving assignment exists for each nondegenerate fundamental class that satisfies the conditions in Lemma \ref{rev_irrev2s}
\end{itemize}
which together satisfy the conditions in Lemma \ref{rev_irrev4} in terms of ${{g_W}\left( {{y_i} \to {y_i}'} \right)}$, ${{h_W}\left( {{y_i} \to {y_i}'} \right)}$, and ${{\rho_W}\left( {{y_i} \to {y_i}'} \right)}$.
\label{main_thm}
\end{theorem}

\section{Applications of the MSA to GMA systems}
\label{applications:gma}
This section emphasizes how powerful the MSA is, as we present a solution to the problem of monostationarity, for any set of rate constants, of the model of anaerobic fermentation pathway in yeast, and the multistationarity, for particular set of rate constants, of the global carbon cycle model of Heck et al. which are not yet known in literature. Using the algorithm, we can determine whether a PL-RDK, with underlying network of any deficiency, has the capacity to admit multiple steady states.

\subsection{Application to the Model of Anaerobic Fermentation Pathway in Yeast}
\label{shorten:yeast}
The fermentation pathway in a species of yeast, known as {\it Saccharomyces cerevisiae}, has been studied extensively. In particular, Galazzo and Bailey \cite{galazzo,galazzo2} established the experimental basis, and they were able to provide kinetic equations which was used by Curto, Cascante and Sorribas \cite{curto} to derive GMA (and S-system) models and performed standard procedures of biochemical systems analysis. 
It was used as a case study in Chapter 8 of \cite{voit} which demonstrates different numerical analyses for biochemical modeling but was not focused on the algebraic aspects of steady-state analyses.

The fermentation pathway is a popular example to be found in numerous BST papers, yet the basic question of its monostationarity or multistationarity is unknown. Different experimental set-ups and conditions under which yeast cells produce ethanol have been investigated which is of high importance for industrial purposes. The pathway in Figure \ref{yeast} describes how yeast can use glucose to produce ethanol, and also glycerol, glycogen and trehalose \cite{voit}.
\begin{table}
\caption{Network Numbers}
\label{network:numbers1}       
\begin{tabular}{lrr}
\noalign{\smallskip}\hline\noalign{\smallskip}
& ERM0-G  & Global Carbon\\
&   & Cycle Model \\
\noalign{\smallskip}\hline\noalign{\smallskip}
species & 5 & 5\\
complexes & 13 & 14\\
reactant complexes & 8 & 9\\
reactions & 13 & 10\\
irreversible reactions & 13 & 6\\
linkage classes & 1 & 6\\
strong linkage classes & 13 & 12\\
terminal sl classes & 5 & 6\\
rank of network & 5 & 4\\
deficiency & 7 & 4\\
\noalign{\smallskip}\hline\noalign{\smallskip}
\end{tabular}
\end{table}
We derive a novel result about the GMA model of anaerobic fermentation pathway of yeast: its monostationarity, i.e., it has at most one steady state for any set of rate constants. The reaction network is given in Illustration \ref{sacc:cer} and its network numbers is provided in Table \ref{network:numbers1} \cite{arceo}. The $T$-matrix is given below.\\
{\small
$\bordermatrix{%
& X_2 & X_1 + X_5 & 2X_5+X_1 & X_2+X_5 & 2X_5+X_2 &  X_3 + X_5 & X_3 +X_4 + X_5 & 2X_5\cr
X_1 & 0 & 0.7464 & 0.7464  & 0 & 0 & 0& 0 & 0\cr
X_2 & -0.2344 & 0 & 0  & 8.6107 & 0.7318 & 0& 0 & 0\cr
X_3 & 0 & 0 & 0  & 0 & 0 & 0.6159& 0.05 & 0\cr
X_4 & 0 & 0 & 0  & 0 & 0 & 0& 0.533 & 0\cr
X_5 & 0 & 0.0243 & 0.0243  & 0 & -0.3941 & 0.1308& -0.0822 & 1\cr
}$
}

{\bf STEP 1:} CHOOSING AN ORIENTATION

In this step, we refer to Illustration \ref{sacc:cer}.

{\bf STEP 2:} FINDING EQUIVALENCE CLASSES AND FUNDAMENTAL CLASSES

In this step, we refer to Illustration \ref{sacc:cer2}. If one of these statements (a) and (b) is not satisfied, then the system does not have the capacity to admit multiple equilibria, and we exit the algorithm.
\begin{itemize}
\item[(a)] All reactions in $P_0$ are reversible (with respect to $\mathscr{R}$).
\item[(b)] For two irreversible reactions (with respect to $\mathscr{R}$), $y \to y'$ and ${\overline y  \to \overline y '}$ in the same $P_i$, there exists $\alpha > 0$ such that $v_{y \to y'}^l = \alpha v_{\overline y  \to \overline y '}^l$ for all $1 \le l \le d$.
\end{itemize}

{\bf STEP 3:} FINDING THE COLINKAGE SETS 

In this step, we divide the reaction network into subnetworks in such a way that all reactions belonging from the same fundamental class are in the same subnetwork. We again refer to Table \ref{tab:PiCiEMR0G}.

{\bf STEP 4:} PICKING $W \subseteq {\mathscr{O}}$

Recall that an equivalence class is reversible if all of its reactions are reversible with respect to the original network $\mathscr{R}$. However, it is nonreversible, if it contains an irreversible reaction.
We pick a representative reaction for each of the $P_i$'s such that if a class $P_i$ is nonreversible, we pick an irreversible reaction. Otherwise, we pick any reversible reaction.
The collection of all the representatives from $P_i$ where $i=1,2,...,w$ is the set $W$ with $w$ as its number of elements.

In our example, since there are 11 equivalence classes, $w=11$. For $P_{1}$, we choose $R_{1}$. Similarly, we choose $R_{8}$ for $P_{7}$. Since the rest of each $P_i$ has only one element, we have no choice but to choose these reactions as representatives. The $i^{\rm th}$ reaction in $W$ is identified as ${y_i} \to {y_i}'$. For instance, $R_{8}$ will be identified as $y_{7} \to y_{7}'$.

{\bf STEP 5:} REALIGNING THE ORIENTATION (if needed)

For each nonzeroth equivalence class $P_i$ with $1 \le i \le w$, for any reaction ${y \to y'}$ in $P_i$, there exists an $\alpha_{y \to y'} >0$ such that $v_{y_i \to y'_i}^l = \alpha_{y \to y'} v_{y  \to  y '}^l$ for all the basis elements $v_1, v_2,...,v_d$ (for $KerL_{\mathscr{O}}$). If this statement does not hold, we then realign the orientation (or choose another orientation) until it is already satisfied. In the given example, the statement is satisfied so we go to the next step.

{\bf STEP 6:} FINDING A BASIS FOR $Ker^{\perp}L_{\mathscr{O}}\cap \Gamma_W$

To simplify this step, from STEP 2, we just consider the rows of the reactions in $W$. Moreover, a basis for $Ker^{\perp}L_{\mathscr{O}}\cap \Gamma_W$ is also given below.
$\bordermatrix{%
& v_1' & v_2' & v_3' & v_4' & v_5' &  v_6' & v_7' & v_8'\cr
w=1 & 0 & 1 & 0  & 0 & 1 & 1& 0 & 0\cr
w=2 & -1 & 0 & -1  & 1 & 0 & 0& 1 & -1\cr
w=3 & 0 & 0 & 0  & 0 & 1 & 1& 0 & 0\cr
w=4 & 1 & 0 & 0  & 0 & 0 & 0& 0 & 0\cr
w=5 & 0 & 1 & 0  & 0 & 0 & 0& 0 & 0\cr
w=6 & 0 & 0 & 1  & 0 & 0 & 0& 0 & 0\cr
w=7 & 0 & 0 & 0  & 0 & 0 & 1& 0 & 0\cr
w=8 & 0 & 0 & 0  & 1 & 0 & 0& 0 & 0\cr
w=9 & 0 & 0 & 0  & 0 & 1 & 0& 0 & 0\cr
w=10 & 0 & 0 & 0  & 0 & 0 & 0& 1 & 0\cr
w=11 & 0 & 0 & 0  & 0 & 0 & 0& 0 & 1\cr
}$
~$\bordermatrix{%
& a_1 & a_2 & a_3 \cr
w=1 & -1 & 0 & 0  \cr
w=2 & 0 & 0 & 1  \cr
w=3 & 1 & -1 & 0  \cr
w=4 & 0 & 0 & 1  \cr
w=5 & 1 & 0 & 0  \cr
w=6 & 0 & 0 & 1  \cr
w=7 & 0 & 1 & 0  \cr
w=8 & 0 & 0 & -1  \cr
w=9 & 0 & 1 & 0  \cr
w={10} & 0 & 0 & -1  \cr
w={11} & 0 & 0 & 1  \cr
}$

{\bf STEP 7:} CHECKING THE LINEARITY OF THE SYSTEM OF INEQUALITIES

If there exists a forest basis for $Ker^{\perp}L_{\mathscr{O}}\cap \Gamma_W$, then the resulting inequality system is linear. If it does not exist, we may need additional nonlinear equations on the $M_i$'s  to determine if the kinetic system has the capacity to admit multiple equilibria.

{\bf STEP 8:} CHOOSING SIGNS FOR ${g_W},{h_W} \in {\mathbb{R}}^{\mathscr{O}} \cap {\Gamma _W}$

Since each $P_i$ is nonreversible, the sign patterns for ${g_W}$ and ${h_W}$ must be positive.

{\bf STEP 9:} SHELVING REACTIONS IN THE NONDEGENERATE $C_i$'s

We refer to Illustration \ref{sacc:cer3}.

{\bf STEP 10:} SHELVING EQUALITIES AND INEQUALITIES FROM THE NONDEGENERATE $C_i$'s

We refer to Lemma \ref{rev_irrev4}. For each nondegenerate $C_i$, if ${{\rho _W}\left( {{y_i} \to {y_i}'} \right)}>0$, let ${M_i} = \ln \left( {{\rho _W}\left( {{y_i} \to {y_i}'} \right)} \right)$. Otherwise, let ${M_i}$ to be an arbitrary large and negative number. Suppose $C_i$ is a nondegenerate fundamental class. If $y \to y' \in {C_i}$ is on the middle shelf, then $T_{.y} \cdot \mu  = {M_i}$ is added to the system. If the given reaction is on the upper shelf, then $T_{.y} \cdot \mu  > {M_i}$ is added to the system. If the given reaction is on the lower shelf, then $T_{.y} \cdot \mu  < {M_i}$ is added to the system. Now, from a reversible reaction $y \to y'$ of $P_0$, it is automatic that $T_{.y} \cdot \mu  = T_{.y'} \cdot \mu $ should be added to the system.\\
$-0.2344{\mu _{X_2}}=M_1=M_2=0.7464{\mu _{X_1}}+0.0243{\mu _{X_5}}$\\
$0.7318{\mu _{X_2}}-0.3941{\mu _{X_5}}=M_3=M_4$\\
$8.6107{\mu _{X_2}}=M_5=M_6$\\
$0.6159{\mu _{X_3}}+0.1308{\mu _{X_5}}=M_7=M_8=M_9=M_{10}=0.05{\mu _{X_3}}+0.533{\mu _{X_4}}-0.0822{\mu _{X_5}}$\\
${\mu _{X_5}}=M_{11}$

{\bf STEP 11:} UPPER AND LOWER SHELVING INEQUALITIES FROM  $P_i$'s WITH NONDEGENERATE $C_i$'s

Suppose $C_i$ is nondegenerate and let $y \to y'$ in $P_i$. If ${g_W}\left( {y_i \to y_i'} \right) > 0$ and ${y \to y'}$ is on the upper shelf, or if ${g_W}\left( {y_i \to y_i'} \right) < 0$ and ${y \to y'}$ is on the lower shelf, then we add $T_{.y} \cdot \mu  < T_{.y'} \cdot \mu $ to the system. However, if ${g_W}\left( {y_i \to y_i'} \right) > 0$ and ${y \to y'}$ is on the lower shelf, or if ${g_W}\left( {y_i \to y_i'} \right) < 0$ and ${y \to y'}$ is on the upper shelf, then we add $T_{.y} \cdot \mu  > T_{.y'} \cdot \mu $ to the system.

We skip this step since the upper and the lower shelves are both empty.

{\bf STEP 12:} ADDING EQUALITIES AND INEQUALITIES FROM $P_i$'s WITH DEGENERATE $C_i$'s

Suppose $C_i$ is degenerate and let $y \to y'$ in $P_i$. If ${h_W}\left( {y_i \to y_i'} \right) > 0$ then we add $T_{.y} \cdot \mu  > T_{.y'} \cdot \mu $ to the system. However, if ${h_W}\left( {y_i \to y_i'} \right) = 0$ then we add $T_{.y} \cdot \mu  = T_{.y'} \cdot \mu $ to the system. Otherwise,  we add $T_{.y} \cdot \mu  < T_{.y'} \cdot \mu $ to the system.

Our example has no degenerate $C_i$.

{\bf STEP 13:} ADDING $M$ EQUALITIES AND INEQUALITIES

From STEP 6, we obtained a basis for $Ker^{\perp}L_{\mathscr{O}}\cap \Gamma_W$. In particular, $b^1=(-1 \ 0 \ 1 \ 0 \ 1 \ 0 \ 0 \ 0 \ 0 \ 0 \ 0)$, $b^2=(0 \ 0 \ -1 \ 0 \ 0 \ 0 \ 1 \ 0 \ 1 \ 0 \ 0)$, and $b^3=(0 \ 1 \ 0 \ 1 \ 0 \ 1 \ 0 \ -1 \ 0 \ -1 \ 1)$. From $b^1$, we have $R^1_+=\{y_3\to y_3',y_5\to y_5'\}$ and $R^1_-=\{y_1\to y_1'\}$. From $b^2$, we have $R^2_+=\{y_7\to y_7',y_9\to y_9'\}$ and $R^2_-=\{y_3\to y_3'\}$. From $b^3$, we have $R^3_+=\{y_2\to y_2',y_4\to y_4',y_6\to y_6',y_{11}\to y_{11}'\}$ and $R^3_-=\{y_8\to y_8',y_{10}\to y_{10}'\}$.

Since our example has no degenerate $C_i$, we only consider the ``nonsegregated case". We take the required sum as zero if the degenerate set is empty. Recall that two multisets $Q_1$ and $Q_2$ are nonsegregated if at least one of the following two cases holds:
\begin{itemize}
\item[i.] There exist $a$ from $Q_1$ or $Q_2$, and $b<c$ from the other such that $b<a<c$.
\item[ii.] All elements in the two multisets are equal, or there exist $a,b\in Q_1$ and $c,d\in Q_2$ such that $c=a<b=d$.
\end{itemize}

From STEP 10, we already have $M_1=M_2$, $M_3=M_4$, $M_5=M_6$, and $M_7=M_8=M_9=M_{10}$. From $b^2$, we obtain $M_7 = M_3 = M_9$.\\
CASE 1: From $b^1$, we set $M_3 < M_1 < M_5$. Hence, from $b^3$, it can only be $M_2 > M_8 > M_{11}$ or $M_6 > M_8 > M_{11}$.\\
CASE 2: From $b^1$, we set $M_3 > M_1 > M_5$. Thus, from $b^3$, it can only be $M_2 < M_8 < M_{11}$ or $M_6 < M_8 < M_{11}$.\\
CASE 3: From $b^1$, we set $M_3 = M_1 = M_5$. Finally, from $b^3$, $M_i=M_j$ for all $i,j\in\{1,2,...,11\}$.

{\bf STEP 14:} CHECKING FOR SOLUTION OF THE SYSTEM

If the system of equations and inequalities is linear, then the system obtained from STEP 10 to STEP 13 is complete. If it has a solution, then it is called a signature. However, if the system is nonlinear, then additional nonlinear constraints are needed to make it complete. If a partial linear system we obtained from STEP 10 to STEP 13 has a solution, then it is called a pre-signature. If none of such inequality systems has a solution, then additional nonlinear constraints are not needed, and we conclude that the system does not have the capacity to admit multiple equilibria.

This step checks when one obtains a $\mu$ which is sign-compatible with the stoichiometric subspace $S$. In other words, we check if there is a linear combination of a basis for $S$ that has the same sign with $\mu$.

At this point, we have the following system.\\
$-0.2344{\mu _{X_2}}=M_1=0.7464{\mu _{X_1}}+0.0243{\mu _{X_5}}$\\
$0.7318{\mu _{X_2}}-0.3941{\mu _{X_5}}=M_7$\\
$8.6107{\mu _{X_2}}=M_5$\\
$0.6159{\mu _{X_3}}+0.1308{\mu _{X_5}}=M_7=0.05{\mu _{X_3}}+0.533{\mu _{X_4}}-0.0822{\mu _{X_5}}$\\
${\mu _{X_5}}=M_{11}$\\
Ignoring the last equation, we obtain the following values.\\
$\mu_{X_1}=   (3442493125M_1)/2154694458 + (101250M_7)/1225651$\\
$\mu_{X_2}=   -(1250M_1)/293$\\
$\mu_{X_3}=    (3988310000M_1)/2370625789 + (52490000M_7)/24272619$\\
$\mu_{X_4}=  (16585421000M_7)/12937305927 - (134085884500M_1)/97195657349$\\
$\mu_{X_5}=     - (9147500M_1)/1154713 - (10000M_7)/3941$\\
By considering any of the cases in STEP 13, the system yields a contradiction. Since the system is inconsistent, there is no need to put additional nonlinear equations from STEP 7.
Therefore, we cannot find a signature $\mu$ and the system does not have the capacity to admit multiple equilibria. We exit the algorithm.

\begin{remark}
In ERM0-G, not only the network properties are essential for the algorithm but we are also particular with the kinetic order values of the kinetic system. From STEPS 1 to 9, we use the structure of the reaction network alone. Within these steps, the orientation is partitioned into equivalence classes. With respect to this partitioning, consequently, the reaction set is partitioned into fundamental classes. Until the assignment of sign patterns for the fundamental classes and the shelving assignment of the reactions in these classes, there is no use of the kinetic order values in the $T$-matrix. The use of these values takes place from STEPS 10 to 14 when we form the system of equations/inequalities until we solve such system.
\end{remark}

\subsection{Application to the Heck et al.'s Global Carbon Cycle Model}
\label{ex:heck}
The global carbon cycle model of Heck et al. \cite{heck} is built from the well-cited box model of Anderies et al. \cite{anderies}. The latter was calibrated according to empirically observed and simulated carbon cycle dynamics. The revised model also incorporates a societal intervention called terrestrial carbon dioxide removal (tCDR). In tCDR, terrestrial carbon is deliberately recovered and dumped into a new sink, referred as carbon engineering sink. A detailed discussion of the model calibration and assumptions can be found in \cite{heck}. 

In the CRN-based analysis of the power-law approximation of this Earth system, Fortun et al. \cite{fortun4} considered pooling the geological carbon pool and the new sink to form a passive carbon pool. Moreover, they separated the atmospheric carbon into two nodes: atmospheric carbon in the pre-industrial state and additional atmospheric carbon due to fossil fuel use. A biochemical map of the system is shown in Figure \ref{heck:pic} \cite{fortun4}. 
\begin{figure}
\begin{center}
\label{heck:pic}
\includegraphics[width=10cm,height=20cm,keepaspectratio]{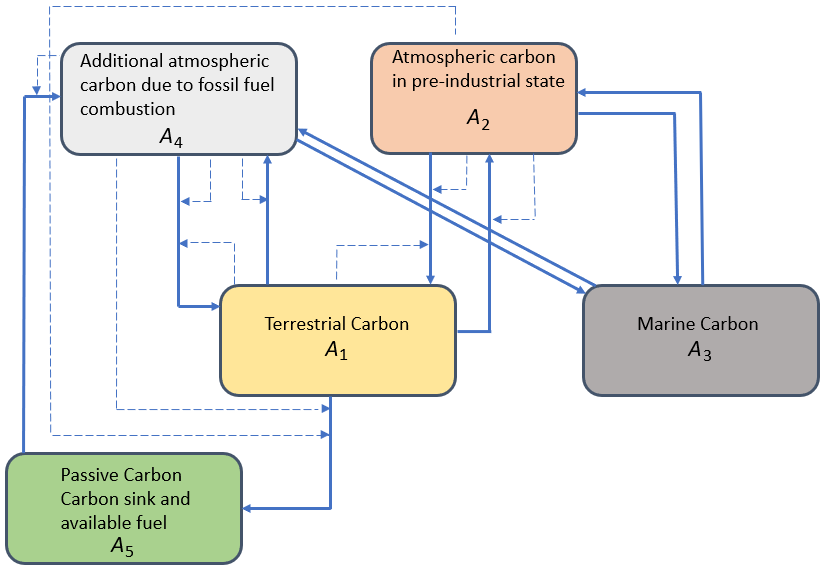}
\caption{Biochemical map of Heck et al.'s global carbon cycle model \cite{fortun4}.}
\end{center}
\end{figure}
The list of network numbers is provided in Table \ref{network:numbers1}. Because of the high deficiency of the network, it becomes a good candidate for the application of MSA to determine the capacity of multiple positive steady states of the system. Multistationarity in global climate implies that there may exist ``tipping points'' beyond which a return to the original state may be difficult or prolonged. By showing that multistationarity in a global climate model may exist, therefore, makes the search for the tipping point relevant. After the tipping point is determined, appropriate measures may then be set to avoid exceeding it.

The reaction network of the Heck et al.'s global carbon cycle model is given below \cite{fortun4}.
$$ \begin{array}{lll}
R_1: A_1 +2A_2 \to 2A_1 + A_2 & \ \ \ & R_6: A_1 + 2A_4  \to 2A_1 + A_4\\
R_2: A_1 + A_2  \to 2A_2  & \ \ \ & R_7: A_1 + A_4  \to 2A_4 \\
R_3: A_2  \to A_3 & \ \ \  & R_8: A_4  \to A_3\\
R_4: A_3  \to A_2  & \ \ \  & R_9: A_3 \to A_4  \\
R_5: A_4 + A_5  \to 2A_4 & \ \ \  & R_{10}: A_1 + A_2 + A_4 \to A_5 + A_2+A_4\\
\end{array}$$
The kinetic system has reactant-determined kinetics and the $T$-matrix is given below.\\
{\small
$\bordermatrix{%
& A_1+2A_2 & A_1 + A_2 & A_2 & A_3 & A_4+A_5 &  A_1 + 2A_4 & A_1 +A_4 & A_4 & A_1+A_2+A_4\cr
A_1 & 199.75 & 159.84 & 0  & 0 & 0 & -43.80& -56.13 & 0&1\cr
A_2 & -86.03 & -63.32 & 1  & 0 & 0 & 0& 0 & 0 &4.44\cr
A_3 & 0 & 0 & 0  & 1 & 0 & 0& 0 & 0&0\cr
A_4 & 0 & 0 & 0  & 0 & 1 & 21.42& 22.19 & 1&11.52\cr
A_5 & 0 & 0 & 0  & 0 & 1.54 & 0& 0 & 0&0\cr
}$
}
Choose ${\mathscr{O}}=\{R_1,R_2,R_4,R_5,R_6,R_7,R_8,R_{10}\}.$
Below is basis for $Ker{L_\mathscr{O}}$ obtained by solving $\sum\limits_{y \to y' \in \mathscr{O}} {{\alpha _{y \to y'}}\left( {y' - y} \right)}=0$.
\[\bordermatrix{%
& v^1 & v^2 & v^3 & v^4\cr
R_1 & 1 & 0 & 1  & 0 \cr
R_2 & 1 & 0 & 0  & 0 \cr
R_4 & 0 & 0 & 1  & 0 \cr
R_5 & 0 & 0 & 0  & 1 \cr
R_6 & 0 & 1 & -1  & 1 \cr
R_7 & 0 & 1 & 0  & 0 \cr
R_8 & 0 & 0 & 1  & 0 \cr
R_{10} & 0 & 0 & 0  & 1 \cr
}\]

A basis for $Ker^{\perp}L_{\mathscr{O}}\cap \Gamma_W$ is also given below.
\[\bordermatrix{%
& v_1' & v_2' & v_3' & v_4' \cr
w=1 & 1& 0 & 1 & 0 \cr
w=2 & 1 & 0 & 0  & 0\cr
w=3 & 0 & 0 & 1  & 0 \cr
w=4 & 0 & 0 & 0  & 1\cr
w=5 & 0 & 1 & -1  & 1\cr
w=6 & 0 & 1	 & 0  & 0 \cr
}
\bordermatrix{%
& b_1 & b_2  \cr
w=1 & -1 & 0  \cr
w=2 & 1 & 0  \cr
w=3 & 1 & 1   \cr
w=4 & 0 & -1   \cr
w=5 & 0 & 1  \cr
w=6 & 0 & -1  \cr
}\]

We can verify that our example has a forest basis for $Ker^{\perp}L_{\mathscr{O}}\cap \Gamma_W$. We will not present the details here. Without this fact, one can obtain a system of equations/inequalities up to STEP 14 and try if $\kappa$ can be solved after getting the value of a solution $\mu$. We pick the sign patterns to be positive and obtain the following shelving assignment. If any of ${{\cal U}_i}$, ${{\cal M}_i}$, or ${{\cal L}_i}$ does not appear in the list, then it is empty.\\
${{\cal M}_1} = \left\{ { R_1 : A_{1}+2 A_{2}\to 2A_1 + A_{2}    } \right\}$\\
${{\cal M}_2} = \left\{ { R_2: A_{1}+A_{2}\to 2A_{2}    } \right\}$\\
${{\cal L}_3} \ = \left\{ {R_3: A_2  \to A_3, R_4: A_3  \to A_2, 
R_8: A_4  \to A_3,R_9: A_3 \to A_4 
} \right\}$\\
${{\cal M}_4} = \left\{ {  R_5: A_{4} + A_{5}  \to 2A_4, R_{10}: A_1 + A_2 + A_4 \to A_5 + A_2+A_4 } \right\}$\\
${{\cal M}_5} = \left\{ {   R_6 :  A_{1}+2 A_{4}\to 2A_1 + A_{4} } \right\}$\\
${{\cal M}_6} = \left\{ {  R_7 :  A_{1}+ A_{4}\to 2A_{4} } \right\}$

We choose $M_2 < M_1 < M_3$ from $b^1$ and $M_5 < M_6 < M_3$ from $b^2$.
Thus, we obtain the following system.
\[ \begin{array}{c}
199.75{\mu _{{A_1}}} - 86.03{\mu _{{A_2}}} = {M_1}\\
159.84{\mu _{{A_1}}} - 63.32{\mu _{{A_2}}} = {M_2} \\
{\mu _{{A_4}}} + 1.54{\mu _{{A_5}}} = {M_4} = {\mu _{{A_1}}} + 4.44{\mu _{{A_2}}} + 11.52{\mu _{{A_4}}}\\
 - 43.80{\mu _{{A_1}}} + 21.42{\mu _{{A_4}}} = {M_5}\\
 - 56.13{\mu _{{A_1}}} + 22.19{\mu _{{A_4}}} = {M_6}\\
{\mu _{{A_2}}}<{\mu _{{A_3}}}<{\mu _{{A_4}}} < {M_3}\\
M_2 < M_1 < M_3\\
M_5 < M_6 < M_3
\end{array}\]
Since the system of inequalities has a nonzero solution $\mu$, the kinetic system has the capacity for multistationarity. A solution of the system together with the multiple steady states are given in Table \ref{tab:equilibria:ex:heck}. The computation of a particular set of rate constants is given in Appendix \ref{rate:constant:heck}.
\begin{table}
\caption{Equilibria in Heck et al.'s Global Carbon Cycle Model}
\label{tab:equilibria:ex:heck}       
\centering
\begin{tabular}{ccccc}
\noalign{\smallskip}\hline\noalign{\smallskip}
& $\mu$ & $\sigma \in S$ &  $c$** & $c$*\\
\noalign{\smallskip}\hline\noalign{\smallskip}
$A_1$& $-0.062895375$ &  $-1$  &  16.40466123 & 15.40466123 \\
$A_2$& $-0.262273058$  & $-1$   &       4.334651228 & 3.334651228 \\
$A_3$&$0.1$  &  1  &       9.508331945 & 10.50833194 \\
$A_4$& 0.291558477  & $1$   &       2.954105867 & 3.954105867 \\
$A_5$&$1.194680148$  &  1  &       0.434310288 & 1.434310288 \\
\noalign{\smallskip}\hline
\end{tabular}
\end{table}

\subsection{A Comparison with the Deficiency One Algorithm for PL-RDK systems}
\label{sect:def_one}
We compare the MSA and the DOA on regular networks. Specifically, we apply the MSA on the main example of \cite{fortun}, the Anderies et al.'s pre-industrial carbon cycle model in \cite{anderies}. 
Indeed, we can verify from Appendix \ref{HDA:anderies} that this deficiency-one model has the capacity for multistationarity for some rate constants \cite{fortun}.

\begin{definition}
A pair of complexes $\{i,j\}$ form a {\bf cut pair} if they are adjacent and the removal of the reaction arrow(s) between these complexes results in a separation of the linkage class containing them.
\end{definition}

Let $\mathscr{C}^1$, $\mathscr{C}^2$, ... , $\mathscr{C}^t$ denote the complex sets of the terminal strong linkage classes
and let $\mathscr{C}'= \bigcup _{k=1}^t \mathscr{C}^i$.

\begin{definition}
A reaction network $\mathscr{N}=\left(\mathscr{S},\mathscr{C},\mathscr{R}\right)$ is said to be {\bf regular} if the following conditions are satisfied.
\begin{itemize}
\item[i.]  The reaction vectors are positively dependent, i.e., there exists a set of positive numbers $\alpha_{ij}$ for all $(i ,j) \in \mathscr{R}$ such that $\sum\limits_{(i , j) \in \mathscr{R}} {{\alpha _{ij}}\left( {j - i} \right)}=0$.
\item[ii.] $\mathscr{N}$ is t-minimal.
\item[iii.] The complexes $i$ and $j$ form a cut pair $\forall \ i , j \in \mathscr{C}'$ such that $(i , j) \in \mathscr{R}$.
\end{itemize}
\end{definition}

Positive dependency is a necessary condition for the existence of a positive equilibrium. In this section, we present an example where MSA can obtain reasonable results for deficiency one, non-$t$-minimal networks and networks without the ``cut pair'' property, which are both outside the scope of the DOA for PL-RDK. 

\begin{example}
\label{ex:cutpairs}
Consider the following reaction network and its kinetic order matrix.
$$\bordermatrix{%
& A_1 & A_2 & A_3 \cr
R_1: A_1+A_3 \to A_1+A_2 & 0.5 & 0 & 0.5  \cr
R_2: A_1 + A_2 \to 2A_3 & 1 & 1 & 0  \cr
R_3: 2A_3 \to A_1+A_3 & 0 & 0 & 1  \cr
R_4: A_3 \to 0 & 0 & 0 & 1  \cr
R_5: A_3 \to A_3+A_2 & 0 & 0 & 1  \cr
}$$
The network has 2 linkage classes and 3 terminal strong linkage classes, and hence, not $t$-minimal. Its deficiency is 1, its rank is 3 and it is not weakly reversible. Moreover, the system is PL-RDK. Also, the complexes $A_1 +A_2$ and $2A_3$ do not form a cut pair. Since neither of last two conditions for regularity of a reaction network is satisfied, the DOA for PL-RDK is not applicable. We will also establish that this system has the capacity to admit multiple steady states using the MSA.

A basis for $Ker L_{\mathscr{O}}$ is
\[\left( {\begin{array}{*{20}{r}}
1&{ - 1}\\
1&0\\
1&0\\
0&1\\
0&1
\end{array}} \right)\]
which induces the following partition: \\${P_1} = \left\{ {{R_1}} \right\} = {C_1}$, ${P_2} = \left\{ {{R_2},{R_3}} \right\} = {C_2}$, and ${P_3} = \left\{ {{R_4},{R_5}} \right\} = {C_3}$.\\
Each reaction is irreversible so we have the following shelving assignment: 
${{\cal M}_1} = {C_1}$, ${{\cal M}_2} = {C_2}$, and ${{\cal M}_3} = {C_3}$. Thus, we have the equations: ${\mu _{{A_1}}} + {\mu _{{A_3}}} = {M_1}$ and ${\mu _{{A_1}}} + {\mu _{{A_2}}} = {M_2} = {M_3} = {\mu _{{A_3}}}$.
From the basis $\{(1,-1,1)\}$ for $Ker^{\perp}L_{\mathscr{O}}\cap \Gamma_W$, we have $M_1=M_2=M_3$ since we already established ${M_2} = {M_3}$. From these equations, we get $\mu=(1,0,1)$ which is stoichiometrically compatible with $\sigma=(3,0,2) \in S$. (You can choose other $\sigma$ which is stoichiometrically compatible with $\mu$, say $\sigma=\mu=(1,0,1)$.) Thus, the kinetic system has the capacity for multistationarity.
We choose $\kappa=(1,2,2,1,1)$, a linear combination of the given basis for $Ker L_{\mathscr{O}}$.

The equilibria $c^{**}$ and $c^{*}$ are given in Table \ref{tab:equilibria:ex:cutpairs}. We refer to Table \ref{tab:values:ex:cutpairs} and check that
$
\sum\limits_{y \to y' \in {\mathscr{R}}} {{\kappa _{y \to y'}}\left( {y' - y} \right) = 0}
$
and
$
\sum\limits_{y \to y' \in {\mathscr{R}}} {{\kappa _{y \to y'}}{e^{{T_{.y }\cdot \mu}}}\left( {y' - y} \right) = 0} 
$
are both satisfied since $\kappa_{y \to y'}e^{T_{.y}\cdot \mu}=k_{y \to y'}c^{*{T_{.y}}}$ is a multiple of the vector $\kappa$. 

\begin{table}
\caption{Equilibria in Example \ref{ex:cutpairs}}
\label{tab:equilibria:ex:cutpairs}       
\centering
\begin{tabular}{ccccc}
\noalign{\smallskip}\hline\noalign{\smallskip}
& $\mu$ & $\sigma \in S$ &  $c$** & $c$*\\
\noalign{\smallskip}\hline\noalign{\smallskip}
$A_1$& 1 &  3  &  1.745930121 & 4.745930121 \\
$A_2$&$0$  & $0$   &       1 & 1 \\
$A_3$&$1$  &  2  &       1.163953414 & 3.163953414 \\
\noalign{\smallskip}\hline
\end{tabular}
\end{table}

\begin{table}
\caption{Summary of values for Example \ref{ex:cutpairs}}
\label{tab:values:ex:cutpairs}         
\centering
\begin{tabular}{cccccc}
\noalign{\smallskip}\hline\noalign{\smallskip}
$y\to y'$&  $\kappa_{y \to y'}$ & $c^{**{T_{.y}}}$ & $c^{*{T_{.y}}}$ & $k_{y \to y'}=\dfrac{\kappa_{y \to y'}}{c^{{**}{T_{.y}}}}$ & $\kappa_{y \to y'}e^{T_{.y}\cdot \mu}$\\
\noalign{\smallskip}\hline\noalign{\smallskip}
 $R_1$    &  1 & 1.425545974 & 3.875035717& 0.701485619 & 2.718281828...$=1e$\\
 $R_2$    &       2 & 1.745930121 & 4.745930121& 1.145521219 & 5.436563657...$=2e$\\
$R_3$    &        2 & 1.163953414 & 3.163953414& 1.718281828 & 5.436563657...$=2e$\\
$R_4$    &        1 & 1.163953414 & 3.163953414& 0.859140914& 2.718281828...$=1e$\\
 $R_5$    &       1 & 1.163953414 & 3.163953414& 0.859140914& 2.718281828...$=1e$\\
\noalign{\smallskip}\hline
\end{tabular}
\end{table}
\end{example}

\section{Multistationarity Algorithm for Power-Law Non-Reactant-Determined Kinetics}
\label{sect:NDK}
Majority of the embedded GMA systems of the 15 identified BST models in \cite{arceo} are PL-NDK. Hence, there is a need of transforming these PL-NDK systems to dynamically equivalent PL-RDK systems. As a result, we combine the CF-RM method and the extension of the HDA to solve the problem of multistationarity of such systems.

\subsection{The CF-RM Method}
\label{chap:six}
\indent In this subsection, we present the CF-RM transformation method (transformation of complex factorizable kinetics by reactant multiples). Further discussion can be found in \cite{cfrm}. From a given PL-NDK system, one can construct a PL-RDK system using this method. In this process, at each reactant complex, the branching reactions are partitioned into CF-subsets (a CF-subset contains reactions having the same kinetic order vectors). 
If a reactant complex has more than one CF-subset, then it is an NF-reactant complex (which makes the system NDK). For each subset, a complex is added to both the reactant and the product complexes of a reaction. This leaves the reaction vectors unchanged. Hence, the stoichiometric subspace remains the same which guarantees the dynamic equivalence of the original and its transform. Under this method, the kinetic order matrix does not change as well.

Let $\mathscr{N}=\left(\mathscr{S},\mathscr{C},\mathscr{R}\right)$ be a reaction network and $\rho : \mathscr{R} \to \mathscr{C}$ be the reactant map. If $y \in \mathscr{C}$, then $\rho ^{-1}(y)$ is its reaction set. Let $\iota: \mathscr{R} \to \mathbb{R}^\mathscr{S}$ be the interaction map of the system. This map assigns to each reaction its kinetic order row in the kinetic order matrix. If $x \in \iota \left( \rho ^{-1}(y) \right)$ then $\{ r \in \rho ^{-1}(y) | \iota (r) =x \}$ is called a CF-subset of $y$. We denote the number of CF-subsets of $y$ by $N_R(y)$. There are $|\iota \left( \rho ^{-1}(y) \right)|$ such subsets. Note that a reactant complex is a CF-reactant complex if and only if $N_R(y)=1$.

The CF-RM method is given by the following steps.
\begin{itemize}
\item [1.] Determine the reactant set $\rho \left(\mathscr{R} \right)$.
\item [2.] Leave each CF-reactant complex unchanged.
\item [3.] At an NF-reactant complex, select a CF-subset containing the highest number of reactions. If there are several such subsets, then just choose one. Leave this CF-subset unchanged.
\item [4.] At this point, there are $N_R(y)-1$ remaining CF-subset. For each of these subsets, choose successively a multiple of $y$ which is not among the current set of reactants.
\end{itemize}

Since we leave each CF-reactant complex unchanged, no new reactant is introduced at a CF-reactant complex. Hence, the total number of new reactants is given by $\sum {\left( {{N_R}\left( y \right) - 1} \right)} $ with the sum taken over all reactants.

We provide an illustration of our computational approach. We show how it can determine multistationarity of a deficiency one PL-RDK.

\subsection{A Deficiency One PL-NDK System}
\label{sect:def:oneLPLNDK}
In this section, we present a weakly reversible PL-NDK system with deficiency one and we apply the combined CF-RM and HDA approach. Typically, the deficiency will increase. Hence, only the MSA can handle the resulting PL-RDK system. This example further emphasizes the advantages of the MSA over the DOA for deficiency one systems.
\begin{example}
\label{ex:cutpairs}
Consider the following reaction network.
$$ \begin{array}{llll}
R_1: 0 \to A_1 & & \ \ \ & R_3: A_1 \to 2A_1\\
R_2: A_1 \to 0  & & \ \ \ & R_4: 2A_1 \to 0 \\
\end{array}$$
The network has 1 linkage class and 1 strong linkage class which is terminal. Its deficiency is 1 and its rank is 1. Moreover, it is weakly reversible. Consider the following kinetic order values: for $R_1: 0$,  for $R_2: 0.5$, for $R_3: 1$, and for $R_4: 0.5$. Hence, the system is PL-NDK. Using the CF-RM transformation, we modify $R_3: 3A_1 \to 4A_1$. The deficiency of the CF-transform is 2. We set $\mathscr{O}=\{R_1,R_3,R_4\}$ and the following are bases for $Ker L_{\mathscr{O}}$ and $Ker L_{\mathscr{R}}$ are
\[\left( {\begin{array}{*{20}{r}}
{ - 1}&2\\
1&0\\
0&1
\end{array}} \right) {\text{ and }} \left( {\begin{array}{*{20}{r}}
1&{ - 1}&2\\
1&0&0\\
0&1&0\\
0&0&1\\
\end{array}} \right), {\text{ respectively. }}\]
The basis for $Ker L_{\mathscr{O}}$ given above induces the following partition: ${P_1} = \left\{ {{R_1}} \right\} \subset {C_1}=\left\{ {{R_1,R_2}} \right\}$, ${P_2} = \left\{ {{R_3}} \right\} = {C_2}$, and ${P_3} = \left\{ {{R_4}} \right\} = {C_3}$.
In $C_1$, consider sign pattern $(+,0)$ for $\left(g_W,h_W\right)$ and the rest are positive since they are nonreversible. Then, we have the following shelving assignment: 
${{\cal U}_1} = {C_1}$, ${{\cal M}_2} = {C_2}$, and ${{\cal M}_3} = {C_3}$ which gives $M_1<0{\mu _{{A_1}}} <0.5{\mu _{{A_1}}}$, $1{\mu _{{A_1}}} = {M_2}$, and $0.5{\mu _{{A_1}}} = {M_3}$.
Now, a basis for for $Ker^{\perp}L_{\mathscr{O}}\cap \Gamma_W$ is $\{(-1,-1,2)\}$ and we choose $M_1<M_3<M_2$. Moreover, take $\kappa=(2,1,1,1)$, a linear combination of the basis elements for $Ker L_{\mathscr{R}}$ so $
\sum\limits_{y \to y' \in {\mathscr{R}}} {{\kappa _{y \to y'}}\left( {y' - y} \right) = 0} 
$.
From the sign pattern of $C_1$ and we recall from our theory (for reversible reactions) that ${g_{0 \to {A_1}}} = {\kappa _{0 \to {A_1}}} - {\kappa _{{A_1} \to 0}} = 2 - 1 > 0$ and ${h_{0 \to {A_1}}} = {\kappa _{0 \to {A_1}}}{e^0} - {\kappa _{{A_1} \to 0}}{e^{0.5{\mu _{{A_1}}}}} = 0$. Hence, $2{e^0} = 1{e^{0.5{\mu _{{A_1}}}}}$ and ${\mu _{{A_1}}} = \ln 4$ which satisfies the obtained system. Choose $\sigma=3 \in S$ so $c^{**}=1$ and $c^{*}=4$. We can verify that
\[f\left( {{c^{**}}} \right) = \left( {\begin{array}{*{20}{c}}
1&{ - 1}&1&-2\\
\end{array}} \right)\left( {\begin{array}{*{20}{c}}
{2{{\left( {1} \right)}^{0}}}\\
{1{{\left( {1} \right)}^{0.5}}}\\
{1{{\left( {1} \right)}^1}}\\
{1{{\left( {1} \right)}^{0.5}}}
\end{array}} \right) 
{\rm \ and \ }
f\left( {{c^*}} \right) = \left( {\begin{array}{*{20}{c}}
1&{ - 1}&1&-2\\
\end{array}} \right)\left( {\begin{array}{*{20}{c}}
{2{{\left( {4} \right)}^{0}}}\\
{1{{\left( {4} \right)}^{0.5}}}\\
{1{{\left( {4} \right)}^1}}\\
{1{{\left( {4} \right)}^{0.5}}}
\end{array}} \right)\]
are both the zero vector.
\end{example}

\section{Conclusions and Outlook}
\label{sec:sum}
We summarize our results and provide some direction for future research.
\begin{itemize}
\item[1.] We modified the higher deficiency algorithm for mass action kinetics of Ji and Feinberg to handle the problem of determining whether a PL-RDK system has the capacity to admit multiple equilibria. This was done by replacing the role of the molecularity matrix by the $T$-matrix.
\item[2.] We applied the algorithm to the embedded network of the GMA model of anaerobic fermentation pathway of {\it Saccharomyces cerevisiae} and determined that it has no capacity to admit multiple equilibria, no matter what positive rate constants are assumed. On the other hand, we determined that the Heck et al.'s global carbon cycle model has the capacity for multistationarity for particular rate constants.
\item[3.] We compared the MSA and the DOA on regular networks by applying it on the Anderies et al.'s pre-industrial carbon cycle model. Indeed, the deficiency-one model has the capacity for multistationarity for some rate constants. This is consistent with the results given in \cite{fortun}. We also presented an example where MSA can obtain results for deficiency one, non-$t$-minimal networks and networks without the ``cut pair'' property since these are outside the scope of the DOA. 
\item[4.] We used the CF-RM transformation to convert a PL-NDK system to a dynamically equivalent PL-RDK  system. We also provided a weakly reversible PL-NDK system with deficiency one and we applied the MSA. Under the CF-RM, the deficiency of a network typically increases. Hence, only the MSA can handle the resulting PL-RDK system which further emphasizes its advantage.
\item[5.] We can also look into the nonlinear case where there is no forest basis and find an example that would be able to apply the algorithm. In addition, one can explore a possible extension of the algorithm to Rate constant-Interaction map-Decomposable (RID) kinetics with interaction parameter maps.
\end{itemize}

\section{Acknowledgement}
BSH acknowledges the support of DOST-SEI (Department of Science and Technology-Science
Education Institute), Philippines for the ASTHRDP Scholarship grant.




\appendix
\section{Nomenclature}
\subsection{List of abbreviations}
\begin{tabular}{ll}
\noalign{\smallskip}\hline\noalign{\smallskip}
Abbreviation& Meaning \\
\noalign{\smallskip}\hline\noalign{\smallskip}
ADA& advanced deficiency algorithm\\
BST&Biochemical Systems Theory\\
CKS& chemical kinetic system\\
CRN& chemical reaction network\\
CRNT& Chemical Reaction Network Theory \\
DOA&deficiency one algorithm  \\
DOT&deficiency one theorem  \\
GMA& generalized mass action\\
HDA& higher deficiency algorithm\\
MAK& mass action kinetics\\
MSA& multistationarity algorithm\\
PLK& power-law kinetics\\
PL-NDK& power-law non-reactant-determined kinetics\\
PL-RDK& power-law reactant-determined kinetics\\
SFRF& species formation rate function\\
\noalign{\smallskip}\hline
\end{tabular}
\subsection{List of important symbols}
\begin{tabular}{ll}
\noalign{\smallskip}\hline\noalign{\smallskip}
Meaning& Symbol \\
\noalign{\smallskip}\hline\noalign{\smallskip}
deficiency& $\delta$  \\
dimension of the stoichiometric subspace& $s$   \\
equivalence class $i$& $P_i$  \\
fundamental class $i$& $C_i$  \\
lower shelf contained in $C_i$& ${\cal L}_i$\\
middle shelf contained in $C_i$& ${\cal M}_i$\\
molecularity matrix& $Y$\\
number of linkage classes& $l$\\
number of strong linkage classes& $sl$\\
number of terminal strong linkage classes& $t$\\
rate constant associated to $y \to y'$& $k_{y \to y'}$\\
stoichiometric matrix& $N$\\
stoichiometric subspace& $S$\\
upper shelf contained in $C_i$& ${\cal U}_i$\\
\noalign{\smallskip}\hline
\end{tabular}

\section{Fundamentals of Chemical Reaction Networks and Kinetic Systems}
\label{prelim}
\indent In this section, we present some fundamentals of chemical reaction networks and chemical kinetic systems. These concepts are provided in \cite{feinberg} and \cite{ji}.

\begin{definition}
A {\bf chemical reaction network} $\mathscr{N}$ is a triple $\left(\mathscr{S},\mathscr{C},\mathscr{R}\right)$ of nonempty finite sets where $\mathscr{S}$, $\mathscr{C}$, and $\mathscr{R}$ are the sets of $m$ species, $n$ complexes, and $r$ reactions, respectively, such that
$\left( {{C_i},{C_i}} \right) \notin \mathscr{R}$ for each $C_i \in \mathscr{C}$; and
for each $C_i \in \mathscr{C}$, there exists $C_j \in \mathscr{C}$ such that $\left( {{C_i},{C_j}} \right) \in \mathscr{R}$ or $\left( {{C_j},{C_i}} \right) \in \mathscr{R}$.
\end{definition}
We can view $\mathscr{C}$ as a subset of $\mathbb{R}^\mathscr{S}_{\ge 0}$. The ordered pair $\left( {{C_i},{C_j}} \right)$ corresponds to the reaction ${C_i} \to {C_j}$.

\begin{definition}
The {\bf molecularity matrix}, denoted by $Y$, is an $m\times n$ matrix such that $Y_{ij}$ is the stoichiometric coefficient of species $X_i$ in complex $C_j$.
The {\bf incidence matrix}, denoted by $I_a$, is an $n\times r$ matrix such that 
$${\left( {{I_a}} \right)_{ij}} = \left\{ \begin{array}{rl}
 - 1&{\rm{ if \ }}{C_i}{\rm{ \ is \ in \ the\ reactant \ complex \ of \ reaction \ }}{R_j},\\
 1&{\rm{  if \ }}{C_i}{\rm{ \ is \ in \ the\ product \ complex \ of \ reaction \ }}{R_j},\\
0&{\rm{    otherwise}}.
\end{array} \right.$$
The {\bf stoichiometric matrix}, denoted by $N$, is the $m\times r$ matrix given by 
$N=YI_a$.
\end{definition}

\begin{definition}
The {\bf reaction vectors} for a given reaction network $\left(\mathscr{S},\mathscr{C},\mathscr{R}\right)$ are the elements of the set $\left\{{C_j} - {C_i} \in \mathbb{R}^\mathscr{S}|\left( {{C_i},{C_j}} \right) \in \mathscr{R}\right\}.$
\end{definition}

\begin{definition}
The {\bf stoichiometric subspace} of a reaction network $\left(\mathscr{S},\mathscr{C},\mathscr{R}\right)$, denoted by $S$, is the linear subspace of $\mathbb{R}^\mathscr{S}$ given by $$S = span\left\{ {{C_j} - {C_i} \in \mathbb{R}^\mathscr{S}|\left( {{C_i},{C_j}} \right) \in \mathscr{R}} \right\}.$$ The {\bf rank} of the network, denoted by $s$, is given by $s=\dim S$. The set $\left( {x + S} \right) \cap \mathbb{R}_{ \ge 0}^\mathscr{S}$ is said to be a {\bf stoichiometric compatibility class} of $x \in \mathbb{R}_{ \ge 0}^\mathscr{S}$.
\end{definition}

\begin{definition}
Two vectors $x, x^{*} \in {\mathbb{R}^\mathscr{S}}$ are {\bf stoichiometrically compatible} if $x-x^{*}$ is an element of the stoichiometric subspace $S$.
\end{definition}

\begin{definition}
A vector $x \in {\mathbb{R}^\mathscr{S}}$ is {\bf stoichiometrically compatible} with the stoichiometric subspace $S$ if 
$sign\left( {{x_s}} \right) = sign\left( {{\sigma _s}} \right)\forall s \in \mathscr{S}$
for some $\sigma \in S$.
\end{definition}

We can view complexes as vertices and reactions as edges. With this, chemical reaction networks can be seen as graphs. At this point, if we are talking about geometric properties, {\bf vertices} are complexes and {\bf edges} are reactions. If there is a path between two vertices $C_i$ and $C_j$, then they are said to be {\bf connected}. If there is a directed path from vertex $C_i$ to vertex $C_j$ and vice versa, then they are said to be {\bf strongly connected}. If any two vertices of a subgraph are {\bf (strongly) connected}, then the subgraph is said to be a {\bf (strongly) connected component}. The (strong) connected components are precisely the {\bf (strong) linkage classes} of a chemical reaction network. The maximal strongly connected subgraphs where there are no edges from a complex in the subgraph to a complex outside the subgraph is said to be the {\bf terminal strong linkage classes}.
We denote the number of linkage classes, the number of strong linkage classes, and the number of terminal strong linkage classes by $l, sl,$ and $t$, respectively.
A chemical reaction network is said to be {\bf weakly reversible} if $sl=l$, and it is said to be {\bf $t$-minimal} if $t=l$.

\begin{definition}
For a chemical reaction network, the {\bf deficiency} is given by $\delta=n-l-s$ where $n$ is the number of complexes, $l$ is the number of linkage classes, and $s$ is the dimension of the stoichiometric subspace $S$.
\end{definition}

\begin{definition}
A {\bf kinetics} $K$ for a reaction network $\left(\mathscr{S},\mathscr{C},\mathscr{R}\right)$ is an assignment to each reaction $j: y \to y' \in \mathscr{R}$ of a rate function ${K_j}:{\Omega _K} \to {\mathbb{R}_{ \ge 0}}$ such that $\mathbb{R}_{ > 0}^\mathscr{S} \subseteq {\Omega _K} \subseteq \mathbb{R}_{ \ge 0}^\mathscr{S}$, $c \wedge d \in {\Omega _K}$ if $c,d \in {\Omega _K}$, and ${K_j}\left( c \right) \ge 0$ for each $c \in {\Omega _K}$.
Furthermore, it satisfies the positivity property: supp $y$ $\subset$ supp $c$ if and only if $K_j(c)>0$.
The system $\left(\mathscr{S},\mathscr{C},\mathscr{R},K\right)$ is called a {\bf chemical kinetic system}.
\end{definition}

\begin{definition}
The {\bf species formation rate function} (SFRF) of a chemical kinetic system is given by $$f\left( x \right) = NK(x)= \displaystyle \sum\limits_{{C_i} \to {C_j} \in \mathscr{R}} {{K_{{C_i} \to {C_j}}}\left( x \right)\left( {{C_j} - {C_i}} \right)}.$$
\end{definition}
The ODE or dynamical system of a chemical kinetics system is $\dfrac{{dx}}{{dt}} = f\left( x \right)$. An {\bf equilibrium} or {\bf steady state} is a zero of $f$.

\begin{definition}
The {\bf set of positive equilibria} of a chemical kinetic system $\left(\mathscr{S},\mathscr{C},\mathscr{R},K\right)$ is given by $${E_ + }\left(\mathscr{S},\mathscr{C},\mathscr{R},K\right)= \left\{ {x \in \mathbb{R}^\mathscr{S}_{>0}|f\left( x \right) = 0} \right\}.$$
\end{definition}

A chemical reaction network is said to admit {\bf multiple equilibria} if there exist positive rate constants such that the ODE system admits more than one stoichiometrically compatible equilibria.

\begin{definition}
A kinetics $K$ is a {\bf power-law kinetics} (PLK) if 
${K_i}\left( x \right) = {k_i}{{x^{{F_{i}}}}} $  $\forall i =1,...,r$ where ${k_i} \in {\mathbb{R}_{ > 0}}$ and ${F_{ij}} \in {\mathbb{R}}$. The power-law kinetics is defined by an $r \times m$ matrix $F$, called the {\bf kinetic order matrix} and a vector $k \in \mathbb{R}^\mathscr{R}$, called the {\bf rate vector}.
\end{definition}
If the kinetic order matrix is the transpose of the molecularity matrix, then the system becomes the well-known {\bf mass action kinetics (MAK)}.

\begin{definition}
A PLK system has {\bf reactant-determined kinetics} (of type PL-RDK) if for any two reactions $i, j$ with identical reactant complexes, the corresponding rows of kinetic orders in $F$ are identical. That is, ${f_{ik}} = {f_{jk}}$ for $k = 1,2,...,m$. A PLK system has {\bf non-reactant-determined kinetics} (of type PL-NDK) if there exist two reactions with the same reactant complexes whose corresponding rows in $F$ are not identical.
\end{definition}

\begin{definition} {\bf \cite{muller}}
\label{tmatrix}
The $m \times n$ matrix $\widetilde Y$ is given by 
\[{\left( {\widetilde Y} \right)_{ij}} = \left\{ \begin{array}{cl}
{\left( F \right)_{ki}} & {{\rm{   if}} \ j \ {\rm{ is \ a \ reactant \ complex \ of \ reaction \ }} k},\\
0  &     {\rm  otherwise}.
\end{array} \right.\]
\end{definition}

\begin{definition}
The $T$-matrix is the $m \times n_r$ truncated $\widetilde Y$ matrix where the nonreactant columns are removed.
\end{definition}

\section{Proofs of the Lemmas and the Main Theorem}
\begin{proof}[of Lemma \ref{converse1}]
We assume positive equilibria ${c}^{*}$ and ${c}^{**}$ are distinct so ${\mu _s} \ne 0$ for some $s \in \mathscr{S}$. Thus, $\mu$ is nonzero. By the monotonicity of $\ln$, $\ln \left( {{c_s}^*} \right) > \ln \left( {{c_s}^{**}} \right)$ whenever ${c_s}^* > {c_s}^{**}$. Hence, ${c_s}^* - {c_s}^{**}$ and $\ln \left( {{c_s}^*} \right) - \ln \left( {{c_s}^*} \right)$ have the same signs. Now, ${\mu _s} = \ln \left( {\dfrac{{{c_s}^*}}{{{c_s}^{**}}}} \right) = \ln \left( {{c_s}^*} \right) - \ln \left( {{c_s}^{**}} \right)$. Since ${c}^*$ and ${c}^{**}$ are stoichiometrically compatible, by definition ${c^*} - {c^{**}} \in S$. Then, $\mu$ is stoichiometrically compatible with $S$. Finally, take ${\kappa _{y \to y'}} = {k_{y \to y'}}{c}{^{**{T_{.y}}}} \ \forall \ y \to y' \in \mathscr{R}$ from the set of positive rate constants  $\left\{ {{k_{y \to y'}}|y' \to y \in \mathscr{R}} \right\}$. 
\end{proof}

\begin{proof}[of Lemma \ref{converse2}]
Since $\mu \ne 0$ then $\mu _s \ne 0$ for some $s\in \mathscr{S}$. Hence, ${\dfrac{{{c_s}^*}}{{{c_s}^{**}}} = {e^{{\mu _s}}} \ne 1}$ so ${c}^*$ and ${c}^{**}$ are distinct. Note that ${c_s}^*$ and ${c_s}^{**}$ must have the same sign and we set both of them to be positive. Let $\sigma \in S$ be stoichiometrically compatible with $\mu$.
If $\mu _s \ne 0$, then we set $c_s^* = \dfrac{{{\sigma _s}{e^{{\mu _s}}}}}{{{e^{{\mu _s}}} - 1}}$ and $c_s^{**} = \dfrac{{{\sigma _s}}}{{{e^{{\mu _s}}} - 1}}$ for all $s \in \mathscr{S}$. In this case, ${c_s}^* - {c_s}^{**} = {\sigma _s}\left( {\dfrac{{{e^{{\mu _s}}} - 1}}{{{e^{{\mu _s}}} - 1}}} \right) = {\sigma _s}$.
If $\mu _s=0$, then we set $c_s^* = c_s^{**} = p$ for some positive number $p$. In this case, ${c_s}^* - {c_s}^{**} = p - p = 0$.
Hence, ${c}^*$ and ${c}^{**}$ are stoichiometrically compatible.
Finally, take ${k_{y \to y'}} = \dfrac{{{\kappa _{y \to y'}}}}{{{c}{{^{**}}^{T_{.y}}}}}.$ 
\end{proof}

\begin{proof}[of Lemma \ref{rev_irrev4}]
Assume $y \to y'$ is irreversible. Then, ${g_{y \to y'}} = {\kappa _{y \to y'}} > 0$, ${h_{y \to y'}} = {\kappa _{y \to y'}}{e^{{T_{.y}} \cdot \mu }} > 0$, and ${\rho _{y \to y'}} = \dfrac{{{h_{y \to y'}}}}{{{g_{y \to y'}}}} = \dfrac{{{\kappa _{y \to y'}}{e^{{T_{.y}} \cdot \mu }}}}{{{\kappa _{y \to y'}}}} = {e^{{T_{.y}} \cdot \mu }}$.
Assume $y \to y'$ is reversible.
Suppose ${g_{y \to y'}} \ne 0$.
Now, ${\rho _{y \to y'}}{g_{y \to y'}} = {h_{y \to y'}}$, ${g_{y \to y'}} = {\kappa _{y \to y'}} - {\kappa _{y' \to y}}$ and ${h_{y \to y'}} = {\kappa _{y \to y'}}{e^{{T_{.y}} \cdot \mu }} - {\kappa _{y' \to y}}{e^{{T_{.y'}} \cdot \mu }}$.
Then, ${g_{y \to y'}}{e^{{T_{.y}} \cdot \mu }} = {\kappa _{y \to y'}}{e^{{T_{.y}} \cdot \mu }} - {\kappa _{y' \to y}}{e^{{T_{.y}} \cdot \mu }}$.
Also, ${\kappa _{y' \to y}}{e^{{T_{.y}} \cdot \mu }} - {\kappa _{y' \to y}}{e^{{T_{.y'}} \cdot \mu }} = {\rho _{y \to y'}}{g_{y \to y'}} - {g_{y \to y'}}{e^{{T_{.y}} \cdot \mu }}$
and ${\kappa _{y' \to y}}\left( {{e^{{T_{.y}} \cdot \mu }} - {e^{{T_{.y'}} \cdot \mu }}} \right) = {g_{y \to y'}}\left( {{\rho _{y \to y'}} - {e^{{T_{.y}} \cdot \mu }}} \right)$. Hence, \\${\kappa _{y' \to y}} = \dfrac{{{\rho _{y \to y'}} - {e^{{T_{.y}} \cdot \mu }}}}{{{e^{{T_{.y}} \cdot \mu }} - {e^{{T_{.y'}} \cdot \mu }}}}{g_{y \to y'}}.$
Similarly, we can solve for the value of ${\kappa _{y \to y'}}$ with
${\kappa _{y \to y'}} = \dfrac{{{\rho _{y \to y'}} - {e^{{T_{.y'}} \cdot \mu }}}}{{{e^{{T_{.y}} \cdot \mu }} - {e^{{T_{.y'}} \cdot \mu }}}}{g_{y \to y'}}.$ Thus, we obtain the following.
\begin{itemize}
\item[i.] If $y \to y'$ is irreversible then ${g_{y \to y'}} > 0$, ${h_{y \to y'}} > 0$, and ${\rho _{y \to y'}} = {e^{{T_{.y}} \cdot \mu }}$.
\item[ii.] Suppose $y \to y'$ is reversible.
\begin{itemize}
\item[a.] If ${g_{y \to y'}} > 0$, then either ${\rho _{y \to y'}} = {e^{{T_{.y}} \cdot \mu }} = {e^{{T_{.y'}} \cdot \mu }}$, ${\rho _{y \to y'}} > {e^{{T_{.y}} \cdot \mu }} > {e^{{T_{.y'}} \cdot \mu }}$, or ${\rho _{y \to y'}} < {e^{{T_{.y}} \cdot \mu }} < {e^{{T_{.y'}} \cdot \mu }}$.
\item[b.] If ${g_{y \to y'}} < 0$, then either ${\rho _{y \to y'}} = {e^{{T_{.y'}} \cdot \mu }} = {e^{{T_{.y}} \cdot \mu }}$, ${\rho _{y \to y'}} > {e^{{T_{.y'}} \cdot \mu }} > {e^{{T_{.y}} \cdot \mu }}$, or ${\rho _{y \to y'}} < {e^{{T_{.y'}} \cdot \mu }} < {e^{{T_{.y}} \cdot \mu }}$.
\item[c.] If ${g_{y \to y'}} = 0$ and ${h_{y \to y'}} > 0$ then ${e^{{T_{.y}} \cdot \mu }} > {e^{{T_{.y'}} \cdot \mu }}$.
\item[d.] If ${g_{y \to y'}} = 0$ and ${h_{y \to y'}} < 0$ then ${e^{{T_{.y}} \cdot \mu }} < {e^{{T_{.y'}} \cdot \mu }}$.
\item[e.] If ${g_{y \to y'}} = {h_{y \to y'}} = 0$ then ${e^{{T_{.y}} \cdot \mu }} = {e^{{T_{.y'}} \cdot \mu }}$.
\end{itemize}
\end{itemize}
Now, let $i \in \{1,2,...,w\}$ and $y \to y' \in P_i$. Also, let ${y_i\to y_{i}'}$ be the representative of $P_i$. Then
${\rho _{y_i \to y_i'}}=\dfrac{{{h_{{y_i} \to {y_i}'}}}}{{{g_{{y_i} \to {y_i}'}}}}=\dfrac{{\alpha _{y_i \to y_i'}}{{h_{{y_i} \to {y_i}'}}}}{{\alpha _{y_i \to y_i'}}{{g_{{y_i} \to {y_i}'}}}}=\dfrac{{{h_{{y} \to {y}'}}}}{{{g_{{y} \to {y}'}}}}={\rho _{y \to y'}}$ for some nonzero ${\alpha _{y_i \to y_i'}}$.
\end{proof}

\begin{proof}[of Lemma \ref{rev_irrev2s}]
By considering the columns of the $T$-matrix instead of the columns of the molecularity matrix, we obtain an analogous proof to to one given in {\cite{ji}}. 
\end{proof}

\begin{lemma} {\cite{ji}} 
Suppose a reaction network satisfies the following properties for an orientation $\mathscr{O}$: $P_i$ ($i=0,1,2,...,w$) is defined by a representative ${{y_i} \to {y_i}'}$, $W=\{{{y_i} \to {y_i}'}|i=1,2,...,w\}\subseteq \mathscr{O}$, a given basis $\left\{ {{b^j}} \right\}_{j = 1}^q$ for $Ker^\perp {L_\mathscr{O}}\cap {\Gamma _W}$ such that its basis graph is a forest, a valid pair of sign patterns for ${h_W},{g_W}$, and a set of parameters\\$\left\{ {{\rho _W}\left( {{y_i} \to {y_i}'} \right)|{g_W}\left( {{y_i} \to {y_i}'} \right) \ne 0,i = 1,2,...,w} \right\}$ where the sign of ${{\rho _W}\left( {{y_i} \to {y_i}'} \right)}$ is the same as the ratio of the signs of ${{h_W}\left( {{y_i} \to {y_i}'} \right)}$ and ${{g_W}\left( {{y_i} \to {y_i}'} \right)}$.
Then ${h_W},{g_W} \in {\mathbb{R}^\mathscr{O}} \cap {\Gamma _W}$ satisfy Equations (\ref{eq:Ker3}), (\ref{eq:Ker4}), and (\ref{eq:Ker5}) if and only if ${\rho _W}\left( {{y_i} \to {y_i}'} \right)$'s satisfy the conditions in Lemma \ref{thm_multisets}.
\label{thm_multisets2}
\end{lemma}

\begin{proof}[of Theorem \ref{main_thm}]
By Lemma \ref{converse1}, (i) is immediate. Since positive and distinct equilibria exist, $Ker L_\mathscr{O}$ is nontrivial. Define $g,h \in \mathbb {R}^\mathscr{O}$ such that\\
${g_{y \to y'}} = \left\{ \begin{array}{ll}
{\kappa _{y \to y'}} - {\kappa _{y' \to y}}&{\rm{      if }} \ y \to y' \in \mathscr{O}{\rm{\  is \ reversible}}\\
{\kappa _{y \to y'}}&{\rm{             if }} \ y \to y' \in \mathscr{O}{\rm{ \  is \  irreversible}}
\end{array} \right.$
and\\
${h_{y \to y'}} = \left\{ \begin{array}{ll}
{\kappa _{y \to y'}}{e^{{T_{.y}} \cdot \mu }} - {\kappa _{y' \to y}}{e^{{T_{.y'}} \cdot \mu }} & {\rm{   if }} \ y \to y' \in \mathscr{O}{\rm{\ is \ reversible}}\\
{\kappa _{y \to y'}}{e^{{T_{.y}} \cdot \mu }} & {\rm{  if }} \ y \to y' \in \mathscr{O}{\rm{\ is \ irreversible}}
\end{array} \right..$\\
For $i=1,2,...,w$, let ${P_i}$ be the equivalence class with ${y_i} \to {y_i}'$ as representative and $M_i$ is associated with the nondegenerate fundamental class.
Let $\{b^j\}_{j=1}^q$ be a basis for $Ker^\perp {L_\mathscr{O}}\cap {\Gamma _W}$, if it exists. Also, let $g,h \in Ker {L_\mathscr{O}}$. Hence,
$b^j \cdot g = \sum\limits_{{y_i} \to {y_i}' \in W} {g_{{y_i} \to {y_i}'} b_{{y_i} \to {y_i}'}^j= 0}$
and
$b^j \cdot h = \sum\limits_{{y_i} \to {y_i}' \in W} {h_{{y_i} \to {y_i}'} b_{{y_i} \to {y_i}'}^j= 0}$ for $j=1,2,...,q$.
These two equations can be written as Equations (\ref{eq:Ker3}), (\ref{eq:Ker4}), and (\ref{eq:Ker5}).
Hence, (iii) which leads to (iv).

On the other hand, by Lemma \ref{converse2}, we only need to solve show the existence of $\kappa$. By Lemma \ref{thm_multisets2} which requires the existence of the forest basis, Equations (\ref{eq:Ker3}), (\ref{eq:Ker4}), and (\ref{eq:Ker5}) hold. From here, we obtain the values of $g_{y \to y'}$ and $h_{y \to y'}$. These yield $\kappa$. 
\end{proof}

\section{Computation Details for Section \ref{applications:gma}}

\subsection{Computation of a Particular Set of Rate Constants for the Heck et al.'s Carbon Cycle Model}
\label{rate:constant:heck}
Table \ref{tab:values:heck} is a summary of values for the Heck et al.'s terrestrial carbon recovery model. The rate constant vector $k_{y \to y'}=\dfrac{\kappa_{y \to y'}}{c^{{**}{T_{.y}}}}$ is indicated. Note that $\kappa_{y \to y'}$ was solved by getting the scalars in the linear combination of a basis for $Ker L_{\mathscr{R}}$ such that $$
\sum\limits_{y \to y' \in {\mathscr{R}}} {{\kappa _{y \to y'}}\left( {y' - y} \right) = 0} 
\label{equi3}
{\rm{ \ and \ }}
\sum\limits_{y \to y' \in {\mathscr{R}}} {{\kappa _{y \to y'}}{e^{{T_{.y }\cdot \mu}}}\left( {y' - y} \right) = 0}.
\label{equi4}
$$
\begin{table}
\caption{Summary of values for for the Heck et al.'s Terrestrial Carbon Recovery Model}
\label{tab:values:heck}         
\centering
\begin{tabular}{cccccc}
\noalign{\smallskip}\hline\noalign{\smallskip}
$y\to y'$&  $\kappa_{y \to y'}$ & $k_{y \to y'}$ & $\kappa_{y \to y'}e^{T_{.y}\cdot \mu}$\\
\noalign{\smallskip}\hline\noalign{\smallskip}
 $R_1$    &  $1.09 \times 10^{-5}$ & $1.4014\times 10^{-193}$ & 0.241022556\\
 $R_2$    &       $1.09 \times 10^{-5}$ & $1.4814\times 10^{-159}$ & 0.00768157\\
$R_3$    &       0.694736046 & 0.160274959 &0.534461087\\
$R_4$    &       0.694736046 & 0.073066028& 0.767802074\\
 $R_5$    &       1 & 1.222817833& 8.426368578\\
 $R_6$    & 1.58138827 & $2.1781\times 10^{43}$& 12814.12188\\
 $R_7$    &      0.58138827 & $3.32525\times 10^{57}$& 12805.92885\\
$R_8$    &      1 & 0.338511903& 1.338511903\\
$R_9$    &        1 & 0.105170918 & 1.105170918\\
 $R_{10}$    &      1 & $3.44864\times 10^{-10}$ & 8.426368578\\
\noalign{\smallskip}\hline
\end{tabular}
\end{table}

\subsection{Application of the MSA on the Anderies et al.'s Pre-industrial Carbon Cycle}
\label{HDA:anderies}
\indent The reaction network and the $T$-matrix of the Anderies et al.'s pre-industrial carbon cycle in \cite{anderies} are given below.
\[\ \ \ \ \ {A_1} + 2{A_2} \to 2{A_1} + {A_2}\]
\[{A_1} + {A_2} \to 2{A_2}\]
\[ \ \ \ \ \ \ {A_2} \mathbin{\lower.3ex\hbox{$\buildrel\textstyle\rightarrow\over
{\smash{\leftarrow}\vphantom{_{\vbox to.5ex{\vss}}}}$}} {A_3}\]
\[\bordermatrix{%
  & A_1+2A_2 & A_1+A_2 & A_2 & A_3\cr
A_1 & {p_1} =  - 1.89& {p_2} =  - 0.27 & 0  & 0 \cr
A_2 & {q_1} = 0.43 & {q_2} = 0.44 & 1  & 0 \cr
A_3 & 0 & 0 & 0  & 1\cr
}\]
We choose the forward reactions so ${\mathscr{O}}=\left\{ {{A_1} + 2{A_2} \to 2{A_1} + {A_2},{A_1} + {A_2} \to 2{A_2},{A_2} \to {A_3}} \right\}.$
A basis for $Ker{L_\mathscr{O}}$, and the corresponding equivalence and fundamental classes are given below.
\[\begin{array}{*{20}{c}}
{{A_1} + 2{A_2} \to 2{A_1} + {A_2}}\\
{{A_1} + {A_2} \to 2{A_2}}\\
{{A_2} \to {A_3}}
\end{array}\left( {\begin{array}{*{20}{c}}
1\\
1\\
0
\end{array}} \right)\]
${P_0} = \left\{ {{A_2}\to{A_3}} \right\} \subset {C_0} = \left\{ {{A_2} \to {A_3},{A_3} \to {A_2}} \right\}$ \\
${P_1} = \left\{ {{A_1} + 2{A_2} \to 2{A_1} + {A_2},{A_1} + {A_2} \to 2{A_2}} \right\}={C_1} $\\
We choose $W = \left\{ {{A_1} + {A_2} \to 2{A_2}} \right\}$ with $w=1$. The resulting inequality system is automatically linear since $Ker^{\perp}L_{\mathscr{O}}\cap \Gamma_W$ is trivial. Since $P_1$ is nonreversible, the sign pattern for ${g_W}$ and ${h_W}$ for ${y_1} \to {y_1}'$ must be positive. We have the following shelving assignment:
${{\cal U}_1} = \left\{  \right\},{{\cal M}_1} = \left\{ {{A_1} + 2{A_2} \to 2{A_1} + {A_2},{A_1} + {A_2} \to 2{A_2}} \right\},{{\cal L}_1} = \left\{ {} \right\}$.
From the middle shelf, we obtain the equation
${p_1}{\mu _{{A_1}}} + {q_1}{\mu _{{A_2}}} = {M_1} = {p_2}{\mu _{{A_1}}} + {q_2}{\mu _{{A_2}}}.$
From $P_0$, we add ${\mu _{{A_2}}} = {\mu _{{A_3}}}$ to the system.
Since the $Ker^{\perp}L_{\mathscr{O}}\cap \Gamma_W$ is trivial, no additional equation or inequality can be obtained here.
We have the following system.
\begin{equation}
\begin{aligned}
{p_1}{\mu _{{A_1}}} + {q_1}{\mu _{{A_2}}} & = {M_1} = {p_2}{\mu _{{A_1}}} + {q_2}{\mu _{{A_2}}}\\
{\mu _{{A_2}}} & = {\mu _{{A_3}}}
\end{aligned}
\end{equation}
Hence, ${\mu _{{A_1}}} = \dfrac{{{q_2} - {q_1}}}{{{p_1} - {p_2}}}{\mu _{{A_2}}}$ and $\mu=\left( \dfrac{{{q_2} - {q_1}}}{{{p_1} - {p_2}}}{\mu _{{A_2}}} ,{\mu _{{A_2}}},{\mu _{{A_2}}} \right)$ with ${\mu _{{A_2}}}>0$.
We take $\left( { - 2,1,1} \right)$ so that $\mu$ is stoichiometrically compatible with $\sigma \in S$. Thus, $\mu$ is a signature and the reaction network has the capacity to admit multiple equilibria. This is consistent with the results of Fortun et al. when they applied the deficiency-one algorithm for PL-RDK in \cite{fortun}.

\subsection{A Remark on the Steps of the MSA}
\begin{remark}
In the case where the first attempt in adding $M$ equations/inequalities has no solution, we introduce STEP 15 where STEPS 13 and 14 are repeated for every choice of $M$ equations/inequalities. A choice is picking one subcase in STEP 13. If there is still no solution, one proceeds with STEP 16 where STEPS 9 to 15 are repeated for possible shelving assignments in STEP 9. If this happens again, we go to STEP 17 where STEPS 8 to 16 are repeated. All sign pattern choices for $g_W,h_W \in \mathbb{R}^\mathscr{O}\cap \Gamma_W$ in STEP 8 are repeated in this step. If there is no signature or pre-signature was found, then the system does not have the capacity to admit multiple stady states, no matter what positive values of the rate constants we assume.
\end{remark}

\section{Additional Useful Propositions}
\label{sect:app_theorem}

\indent The following propositions are useful in determining if a chemical reaction network does not have the capacity to admit multiple steady states if the reaction has at least one irreversible inflow reaction or irreversible outflow reaction.

\begin{proposition}
Let $\left(\mathscr{S},\mathscr{C},\mathscr{R},K\right)$ be a PL-RDK system. Suppose it has an irreversible inflow reaction $0 \to A$. Let $\mathcal D$ be the set of all reactions with $A$ in either reactant or product complex, not including the chosen reaction $0 \to A$. If $A$ does not appear for each simplified reaction vector in $\mathcal D$, then the system does not have the capacity to admit multiple equilibria.

\label{thm_inflow}
\end{proposition}

\begin{proof}
Suppose an irreversible inflow reaction $0 \to A$ exists. By definition of orientation, the reaction $0 \to A$ must be an element of any orientation. Let $\mathscr{O}$ be an orientation. In STEP 2 of the algorithm, we are getting $Ker L_\mathscr{O}$. At this point, we solve for the $\alpha_{y \to y'}$'s given the equation $\displaystyle \sum\limits_{y \to y' \in \mathscr{O}}^{} {{\alpha _{y \to y'}}\left( {y' - y} \right)}=0$. Since after simplifying the reaction vectors, $A$ does not appear in reaction vector $y'-y$ for each reaction $y \to y'$, the reaction $0 \to A$ corresponds to a row with all entries 0. Since $0 \to A$ is irreversible, then it must be placed on the zeroth equivalence class $P_0$. But for a system to have the capacity to admit multiple equilibria, each element in $P_0$ must be reversible (given in STEP 2). 
\end{proof}

\begin{proposition}
Let $\left(\mathscr{S},\mathscr{C},\mathscr{R},K\right)$ be a PL-RDK system. Suppose it has an irreversible outflow reaction $B \to 0$. Let $\mathcal D$ be the set of all reactions with $B$ in either reactant or product complex, not including the chosen reaction $B \to 0$. If $B$ does not appear for each simplified reaction vector in $\mathcal D$, then the system does not have the capacity to admit multiple equilibria.
\label{thm_outflow}
\end{proposition}

\begin{proof}
The proof is similar to the proof of Proposition \ref{thm_inflow}. 
\end{proof}

The propositions above are usually applied for total representation of a reaction network where there are several independent variables. In this representation, for each interaction $X_1 \to X_2$ with regularity arrow from each element of $\{X_j\}$, the reaction $X_1+\sum X_j \to X_1+\sum X_j$ is associated, and any interaction without regulatory arrow are kept as it is \cite{arceo}.
By satisfying the assumptions, one can decide whether a system has the capacity to admit multiple equilibria.
\end{document}